\documentclass[12pt]{article} \textheight 8.8 true in \textwidth
6.2 true in

\usepackage[T2A]{fontenc}
\usepackage[cp1251]{inputenc}
\usepackage[english]{babel}
\usepackage{amsfonts}
\usepackage{amssymb,amsmath,amstext}
\usepackage{amscd, amsthm, latexsym}
\usepackage{color}
\usepackage[final]{graphicx}
\usepackage{epsfig}
\usepackage{euscript}
\usepackage{srcltx,epsf}
\usepackage{wrapfig}
\usepackage{mathrsfs}
\usepackage{ifthen}
\usepackage{cleveref}
\usepackage{multirow}
\usepackage{array}
\usepackage{diagbox}
\usepackage{rotating}

\def\hide#1{}
\def\old#1{}

\def\oop#1{}
\def\gap#1{}

\theoremstyle{plain}
\newtheorem{theorem}{Theorem}
\newtheorem{proposition}{Proposition}
\newtheorem*{theorem*}{Theorem}
\newtheorem{lemma}{Lemma}
\newtheorem{remark}{Remark}
\newtheorem{corollary}{Corollary}
\newtheorem{definition}{Definition}

\def \ZZ  {\Bbb Z}

\begin{document}

\newcounter{figcounter}
\setcounter{figcounter}{0} \addtocounter{figcounter}{1}

\title{On the coverings of closed non-orientable Euclidean manifolds $\mathcal{B}_{3}$ and $\mathcal{B}_{4}$.}
\author{
G.~Chelnokov\thanks{This work was supported by Ministry of Education and Science of the Russian
Federation in the framework of MegaGrant no 075-15-2019-1926}\\
{\small\em National Research University Higher School of Economics, Moscow, Russia} \\
{\small\em Laboratory of Combinatorial and Geometric Structures,}\\
{\small\em Moscow Institute of Physics and Technology, Moscow, Russia}\\
 {\small\tt grishabenruven@yandex.ru }
\\ [2ex]
A.~Mednykh\thanks{This work was supported by the Russian Foundation for Basic Research (grant 16-31-00138).}\\
{\small\em Sobolev Institute of Mathematics, Novosibirsk, Russia} \\
{\small\em Novosibirsk State University, Novosibirsk, Russia}\\
{\small\tt mednykh@math.nsc.ru} }
\date{}
\maketitle 
\begin{abstract}
There are only 10 Euclidean forms, that is  flat closed three
dimensional manifolds: six are orientable
$\mathcal{G}_1,\dots,\mathcal{G}_6$ and four are non-orientable
$\mathcal{B}_1,\dots,\mathcal{B}_4$. The aim of this paper is to
describe all types of $n$-fold coverings over the non-orientable
Euclidean manifolds $\mathcal{B}_{3}$ and $\mathcal{B}_{4}$, and
calculate the numbers of non-equivalent coverings of each type. The
manifolds $\mathcal{B}_{3}$ and $\mathcal{B}_{4}$ are uniquely
determined among non-orientable forms by their homology groups
$H_1(\mathcal{B}_{3})=\ZZ_2\times \ZZ_2 \times \ZZ$ and
$H_1(\mathcal{B}_{4})=\ZZ_4 \times \ZZ$.

We classify subgroups in the fundamental groups
$\pi_1(\mathcal{B}_{3})$ and $\pi_1(\mathcal{B}_{4})$ up to
isomorphism.  Given index $n$,  we calculate the numbers of
subgroups and the numbers of conjugacy classes of subgroups for each
isomorphism type and provide the Dirichlet generating functions for
the above sequences.


\end{abstract}

\section*{Introduction}
Let  $\mathcal{M}$  be a connected manifold with fundamental group
$G=\pi_{1}(\mathcal{M}).$ Two coverings
$$p_1: \mathcal{M}_1 \to \mathcal{M}   \text {      and       } p_2:
\mathcal{M}_2 \to \mathcal{M}  $$ are said to be equivalent if there
exists a homeomorphism $h:  \mathcal{M}_1 \to \mathcal{M}_2$ such
that $p_1 =p_2 \circ h.$ According to the general theory of covering
spaces, any $n$-fold covering is uniquely determined by a subgroup
of index $n$ in the group $G$. The equivalence classes of $n$-fold
coverings of $\mathcal{M}$ are in one-to-one correspondence with the
conjugacy
 classes of subgroups of index $n$ in the fundamental group
 $\pi_1(\mathcal{M}).$ See,
for example, (\cite{Hatch}, p.~67).
  In such a way the following natural problems arise: to
  describe the isomorphism classes of subgroups of finite index in
  the fundamental group of a given manifold and to enumerate the
  finite index subgroups and their conjugacy classes with respect to
  isomorphism type.

We use the following notations: let $s_G(n)$ denote the number of
subgroups of index $n$ in the group $G$, and let $c_G(n)$ be the
number of conjugacy classes of such subgroups. Similarly, by
$s_{H,G}(n)$ denote the number of subgroups of index $n$ in the
group $G$, which are isomorphic to $H$, and by $c_{H,G}(n)$ the
number of conjugacy classes of such subgroups. So, $c_G(n)$
coincides with the number of nonequivalent $n$-fold coverings over a
manifold $\mathcal{M}$ with fundamental group
$\pi_1(\mathcal{M})\cong G$, and $c_{H,G}(n)$ coincides with the
number of nonequivalent $n$-fold coverings $p: \mathcal{N}\to
\mathcal{M}$, where $\pi_1(\mathcal{N})\cong H$ and
$\pi_1(\mathcal{M})\cong G$. The numbers $s_G(n)$ and $c_G(n)$,
where $G$ is the fundamental group of closed orientable or
non-orientable surface, were found  in (\cite{Med78}, \cite{Med79},
\cite{MP86}). In the paper \cite{Medn}, a general method for
calculating the number $c_G(n)$  of conjugacy classes of subgroups
in an arbitrary finitely generated group $G$ was given. Asymptotic
formulas for $s_G(n)$ in many important cases were obtained in
\cite{Lub}.

The values of $s_G(n)$ for the wide class of 3-dimensional Seifert
manifolds were calculated in \cite{LisMed00} and \cite{LisMed12}.
The present paper is a part of the series of our papers devoted to
enumeration of finite-sheeted coverings of coverings over closed
Euclidean 3-manifolds. These manifolds are also known as flat
3-dimensional manifolds or Euclidean 3-forms.

The class of such manifolds is closely related to the notion of
Bieberbach group. Recall that a subgroup of isometries of
$\mathbb{R}^3$ is called {\em Bieberbach group} if it is discrete,
cocompact and torsion free.  Each  $3$-form can be represented as a
quotient $\mathbb{R}^3/G$ where $G$ is a Bieberbach group. In this
case, $G$ is isomorphic to the fundamental group of the manifold,
that is $G\cong\pi_1(\mathbb{R}^3/G)$.  Classification of three
dimensional Euclidean forms up to homeomorphism is presented in
\cite{Wolf}. The class of such manifolds consists of six orientable
$\mathcal{G}_{1}$, $\mathcal{G}_{2}$, $\mathcal{G}_{3}$,
$\mathcal{G}_{4}$, $\mathcal{G}_{5}$, $\mathcal{G}_{6}$, and four
non-orientable ones $\mathcal{B}_{1}$, $\mathcal{B}_{2}$,
$\mathcal{B}_{3}$, $\mathcal{B}_{4}$. One can find the
correspondence between Wolf and Conway-Rossetti notations of these
Euclidean 3-manifolds and their homology groups in Table~1 in
\cite{We1}.


In our previous paper \cite{We1} we describe isomorphism types of
finite index subgroups $H$ in the fundamental group $G$ of manifolds
$\mathcal{B}_1$ and $\mathcal{B}_2$. Further, we calculate the
respective numbers $s_{H,G}(n)$ and $c_{H,G}(n)$ for each
isomorphism type $H$.
In subsequent articles \cite{We2} and \cite{We3} similar questions
were solved for manifolds $\mathcal{G}_2$, $\mathcal{G}_3$
$\mathcal{G}_4$  and $\mathcal{G}_5$.

%

The aim of the present paper is to solve the same questions for
manifolds $\mathcal{B}_3$ and $\mathcal{B}_4$. This manifolds are
uniquely defined among non-orientable Euclidean forms by their
homology groups  $H_1(\mathcal{B}_{3})=\ZZ_2^2\oplus \ZZ$ (coincides
with $H_1(\mathcal{G}_{2})$, but the manifold $\mathcal{G}_{2}$ is
orientable) and $H_1(\mathcal{B}_{4})=\ZZ_4\oplus \ZZ$. To describe
these manifolds through Bieberbach group, consider the following
isometries of $\mathbb{R}^3$:
\begin{align*}
S_1: (x,y,z) &\mapsto (x+1,y,z),\\
S_2: (x,y,z) &\mapsto (-x,y+1,z),\\
S_3: (x,y,z) &\mapsto
(-x,-y,z+1),\\
\widetilde{S}_3: (x,y,z) &\mapsto (-x+1/2,-y,z+1).
\end{align*}

 The Bieberbach groups
$\pi_1(\mathcal{B}_{3})$ and $\pi_1(\mathcal{B}_{4})$ are generated
by triples $S_1,S_2,S_3$ and $S_1,S_2,\widetilde{S}_3$ respectively.

To describe $\mathcal{B}_{3}$ and $\mathcal{B}_{4}$ in more
geometric terms we do the following. Take the cube $[0,1]^3$ in
$\mathbb{R}^3$ (it serves as the fundamental domain for both
manifolds). Glue its faces $x=0$ and $x=1$ by parallel shift
$(x,y,z)\mapsto (x+1,y,z)$. Glue the faces $y=0$ and $y=1$ by mirror
symmetry $(x,y,z)\mapsto (-x+1,y,z)$ followed by parallel shift
$(x,y,z)\mapsto (x,y+1,z)$. Finally, in case of $\mathcal{B}_{3}$
glue the face $z=0$ to face $z=1$ by the central symmetry
$(x,y,z)\mapsto (-x+1,-y+1,z)$ followed with parallel shift
$(x,y,z)\mapsto (x,y,z+1)$. In case of $\mathcal{B}_{4}$ split the
face $z=0$ into two equal rectangles (longer side parallel to $OY$),
self-align each of them by the central symmetry $(x,y,z)\mapsto
(-x+1/2,-y+1,z)$ and $(x,y,z)\mapsto (-x+3/2,-y+1,z)$ respectively.
Then shift the face $z=0$ to $z=1$ with $(x,y,z)\mapsto (x,y,z+1)$.

In the present paper, we classify subgroups in the fundamental
groups $\pi_1(\mathcal{B}_{3})$ and $\pi_1(\mathcal{B}_{4})$ up to
isomorphism. Given index $n$, we calculate the numbers of subgroups
and the numbers of conjugacy classes of subgroups for each
isomorphism type.  Also, we provide the Dirichlet generating
functions for all the above sequences.

Numerical methods to solve these and similar problems for the
three-dimensional crystallographic groups were developed by the
Bilbao group \cite{babaika}. The convenience of language of
Dirichlet generating series for this kind of problems was
demonstrated in \cite{Ruth}. The first homologies of all the
three-dimensional crystallographic groups are determined in
\cite{Ratc}.

\subsection*{Notations}
Suppose $G$ is a group, $u$, $v$ are its elements and $H$, $F$ are
its subgroups. We use $u^v$ instead of $vuv^{-1}$ and $[u,z]$ in
place of $uvu^{-1}v^{-1}$ for the sake of brevity. By $H^v$ denote
the subgroup $\{u^v|\,u\in H\}$. By $H^F$ denote the family of
subgroups $H^v,\, v \in F$. By $Ad_v: G \to G$ denote the
automorphism, given by $u \to u^v$.

By $s_{H,G}(n)$ we denote the number of subgroups of index $n$ in
the group $G$, isomorphic to the group $H$; by $c_{H,G}(n)$ the
number of conjugacy classes of subgroups of index $n$ in the group
$G$, isomorphic to the group $H$. Through this paper usually $G$ and
$H$ are fundamental groups of manifolds $\mathcal{G}_{i}$ or
$\mathcal{B}_{i}$, in this case we omit $\pi_1$ in indexes.

 Also we will need the
following number-theoretic functions. Given a fixed $n$ we widely
use summation over all representations of $n$ as a product of two or
three positive integer factors $\displaystyle \sum_{ab=n}$ and
$\displaystyle\sum_{abc=n}$. The order of factors is important. Also
we consider this sum vanishes if $n$ is not integer.

To start with, this is the natural language to express the function
$\sigma_0(n)$ -- the number of representations of number $n$ as a
product of two factors
$$
\sigma_0(n)=\sum_{ab=n}1.
$$
We will also need the following generalizations of $\sigma_0$:
\begin{align*}
\sigma_1(n)&=\sum_{ab=n}a;
\quad\quad\quad\sigma_2(n)=\sum_{ab=n}\sigma_1(a)=\sum_{abc=n}a;\quad\quad
&
d_3(n)&=\sum_{ab=n}\sigma_0(a)=\sum_{abc=n}1;\\
\chi(n)&=\sum_{ab=n}a\sigma_1(b)=\sum_{ab=n}a\sigma_0(a)=\sum_{abc=n}ab;
& \omega(n) &=\sum_{ab=n}a\sigma_1(a)= \sum_{abc=n}a^2b.
\end{align*}
%
%
%

\section{Formulation of main results}
The main goal of this paper is to prove the following four theorems.

\begin{theorem}\label{th-1-amphidi+}
Every subgroup $\Delta$ of finite index $n$ in
$\pi_{1}(\mathcal{B}_{3})$ have one of the following isomorphism
types: $\pi_1(\mathcal{G}_{1})\cong\ZZ^3$, $\pi_1(\mathcal{G}_{2})$,
$\pi_{1}(\mathcal{B}_{1})$, $\pi_{1}(\mathcal{B}_{2})$,
$\pi_{1}(\mathcal{B}_{3})$ or $\pi_{1}(\mathcal{B}_{4})$. The
respective numbers of subgroups are given by the formulas
\begin{align*}
(i)\quad s_{G_1,B_3}(n)&=\omega(\frac{n}{4}), & (ii)\quad s_{G_2,B_3}(n)&=\omega(\frac{n}{2})-\omega(\frac{n}{4}), \\
(iii)\quad s_{B_1,B_3}(n)&=2\chi(\frac{n}{2})-2\chi(\frac{n}{4}), &
(iv)\quad s_{B_2,B_3}(n)&=4\chi(\frac{n}{4})-4\chi(\frac{n}{8}), \\
(v)\quad s_{B_3,B_3}(n)&= \chi(n) -
3\chi(\frac{n}{2})+2\chi(\frac{n}{4}), & (vi)\quad s_{B_4,B_3}(n)&=
2\chi(\frac{n}{2}) - 6\chi(\frac{n}{4})+4\chi(\frac{n}{8}).
\end{align*}
%
%
\end{theorem}

\begin{theorem}\label{th-2-amphidi+}
The numbers of non-equivalent $n$-fold covering over
$\mathcal{B}_{3}$ with respect to homeomorphism type (i.e. the
numbers of conjugacy classes of subgroups of index $n$ in
$\pi_{1}(\mathcal{B}_{3})$ with respect to isomorphism type of a
subgroup) are given by
$$
c_{G_1,B_3}(n)=\frac{1}{4}\Big(\omega(\frac{n}{4})+3\sigma_2(\frac{n}{4})+9\sigma_2(\frac{n}{8})\Big),
\leqno (i)
$$
$$
c_{G_2,B_3}(n)=\frac{1}{2}\Big(\sigma_2(\frac{n}{2})+2\sigma_2(\frac{n}{4})-3\sigma_2(\frac{n}{8})+d_3(\frac{n}{2})-d_3(\frac{n}{4})+d_3(\frac{n}{8})-3d_3(\frac{n}{16})+2d_3(\frac{n}{32})\Big).\leqno
(ii)
$$
$$
c_{B_1,B_3}(n)=\sigma_2(\frac{n}{2})-\sigma_2(\frac{n}{8})+d_3(\frac{n}{2})-d_3(\frac{n}{8}),\leqno
(iii)
$$
$$
c_{B_2,B_3}(n)=2\sigma_2(\frac{n}{4})-2\sigma_2(\frac{n}{8})+d_3(\frac{n}{4})-d_3(\frac{n}{16}),\leqno
(iv)
$$
$$
c_{B_3,B_3}(n)=d_3(n)-d_3(\frac{n}{2})-d_3(\frac{n}{4})+d_3(\frac{n}{8}),
\leqno (v)
$$
$$
c_{B_4,B_3}(n)=2d_3(\frac{n}{2})-4d_3(\frac{n}{4})+2d_3(\frac{n}{8}).
\leqno (vi)
$$
\end{theorem}

\begin{theorem}\label{th-1-amphidi-}
Every subgroup $\Delta$ of finite index $n$ in
$\pi_{1}(\mathcal{B}_{3})$ have one of the following isomorphism
types: $\pi_1(\mathcal{G}_{1})\cong\ZZ^3$, $\pi_1(\mathcal{G}_{2})$,
$\pi_{1}(\mathcal{B}_{1})$, $\pi_{1}(\mathcal{B}_{2})$ or
$\pi_{1}(\mathcal{B}_{3})$. The respective numbers of subgroups are
\begin{align*}
(i)\quad s_{G_1,B_4}(n)&=\omega(\frac{n}{4}), & (ii)\quad s_{G_2,B_4}(n)&=\omega(\frac{n}{2})-\omega(\frac{n}{4}), \\
(iii)\quad s_{B_1,B_4}(n)&=2\chi(\frac{n}{2})-2\chi(\frac{n}{4}), &
(iv)\quad s_{B_2,B_4}(n)&=4\chi(\frac{n}{4})-4\chi(\frac{n}{8}), \\
(v)\quad s_{B_4,B_4}(n)&= \chi(n) -
5\chi(\frac{n}{2})+8\chi(\frac{n}{4})-4\chi(\frac{n}{8}). &&
\end{align*}
%
\end{theorem}

\begin{theorem}\label{th-2-amphidi-}
The numbers of non-equivalent $n$-fold covering over
$\mathcal{B}_{4}$ with respect to homeomorphism type (i.e. the
numbers of conjugacy classes of subgroups of index $n$ in
$\pi_{1}(\mathcal{B}_{4})$ with respect to isomorphism type of a
subgroup) are
$$
c_{G_1,B_4}(n)=\frac{1}{4}\Big(\omega(\frac{n}{4})+3\sigma_2(\frac{n}{4})+9\sigma_2(\frac{n}{8})\Big),\leqno
(i)
$$
$$
c_{G_2,B_4}(n)=\frac{1}{2}\Big(\sigma_2(\frac{n}{2})+2\sigma_2(\frac{n}{4})-3\sigma_2(\frac{n}{8})+d_3(\frac{n}{2})-d_3(\frac{n}{4})-3d_3(\frac{n}{8})+5d_3(\frac{n}{16})-2d_3(\frac{n}{32})\Big),\leqno
(ii)
$$
$$
c_{B_1,B_4}(n)=\sigma_2(\frac{n}{2})-\sigma_2(\frac{n}{8})+d_3(\frac{n}{2})-2d_3(\frac{n}{4})+d_3(\frac{n}{8}),\leqno
(iii)
$$
$$
c_{B_2,B_4}(n)=2\sigma_2(\frac{n}{4})-2\sigma_2(\frac{n}{8})+d_3(\frac{n}{4})-2d_3(\frac{n}{8})+d_3(\frac{n}{16}),\leqno
(iv)
$$
$$
c_{B_4,B_4}(n)=d_3(n)-3d_3(\frac{n}{2})+3d_3(\frac{n}{4})-d_3(\frac{n}{8}).\leqno
(v)
$$
\end{theorem}

Also we present an alternative proof for the previously known
results (see \cite{Med88}) about the enumeration of subgroups and
conjugacy classes of subgroups of the fundamental group of Klein
bottle.

\begin{theorem}[Klein bottle]\label{enum_all for K}
Let $\pi_1(\mathcal{K})=\langle x,y: yxy^{-1}=x^{-1}\rangle$ be the
fundamental group of Klein bottle. Then each subgroup of finite
index in $\pi_1(\mathcal{K})$ is isomorphic to either
$\pi_1(\mathcal{K})$ or $\ZZ^2$. The respective numbers of subgroups
and conjugacy classes of subgroups are
\begin{align*}
s_{\ZZ^2,\pi_{1}(\mathcal{K})}(n)&=\sigma_1(\frac{n}{2}), &
c_{\ZZ^2,\pi_{1}(\mathcal{K})}(n)&=\frac{1}{2}\Big(\sigma_1(\frac{n}{2})+\sigma_0(\frac{n}{2})+\sigma_0(\frac{n}{4})\Big),\\
s_{\pi_{1}(\mathcal{K}),\pi_{1}(\mathcal{K})}(n)&=\sigma_1(n)-\sigma_1(\frac{n}{2}),
&
c_{\pi_{1}(\mathcal{K}),\pi_{1}(\mathcal{K})}(n)&=\sigma_0(n)-\sigma_0(\frac{n}{4}).
\end{align*}
\end{theorem}

For Dirichlet generating series for the sequences, provided by
Theorems 1--5 see Appendix.

\section{Preliminaries}
Further we use the following representations for the fundamental
groups $\pi(\mathcal{B}_3)$ and $\pi(\mathcal{B}_4)$, see
\cite{Wolf} or \cite{Conway}.
\begin{equation}\label{fund_B3}
 \pi_{1}(\mathcal{B}_{3})=\langle x, y, z: yxy^{-1}=zxz^{-1}=x^{-1},
    zyz^{-1}=y^{-1}
 \rangle .
 \end{equation}
\begin{equation}\label{fund_B4}
 \pi_{1}(\mathcal{B}_{4})=\langle x, y, z: yxy^{-1}=zxz^{-1}=x^{-1},
    zyz^{-1}=xy^{-1}
 \rangle .
 \end{equation}

Below we represent the free abelian groups of rank two and three by
the set of pairs and triples of integer numbers respectively. Given
a subgroup $H$ of an abelian group $G$, we say that two elements
$u,v \in G$ are {\em congruent modulo $H$} if $u-v \in H$. In this
case we write $u \equiv v \mod H$.

We need the following version of Proposition 1 in \cite{We2}.

%

\begin{proposition}\label{enumeration_of sublattices}
\begin{itemize}
\item[a)]The subgroups of index $n$ in $\ZZ^2$ are in one-to-one correspondence with the matrices $\begin{pmatrix}   b & d \\
 0 & a \end{pmatrix}$, where $ab=n$, $0 \le d < a$. A subgroup $\Delta$ of index $n$ is generated by the rows $(0,a)$ and $(b,d)$ of corresponding
 matrix $\Delta=\langle(0,a), (b,d)\rangle$. Consequently,
the number of such subgroups is $\sigma_1(n)$.

\item[b)] The subgroups of index $n$ in $\ZZ^3$ are in one-to-one correspondence with the matrices $\begin{pmatrix} c & e & f \\
0 & b & d \\ 0 & 0 & a\end{pmatrix}$, where $a,b,c >0, \,abc=n$, $0
\le d,f < a$ and $0 \le e < b$.  A subgroup $\Delta$ of index $n$ is
generated by the rows $(0,0,a)$, $(0,b,d)$ and $(c,e,f)$ of
corresponding
 matrix $\Delta=\langle(0,0,a), (0,b,d), (c,e,f)\rangle$. Consequently,
the number of such subgroups is $\sigma_2(n)$.
\end{itemize}
\end{proposition}

%
%
%
%

\begin{corollary}\label{representatives in ZZ}
Let $\Delta$ be a subgroup of finite index $n$ in $\ZZ^2$ and
$\begin{pmatrix}   b & d \\0 & a \end{pmatrix}$ be its corresponding
matrix as described in \Cref{enumeration_of sublattices}. Then the
set of elements $\{(i,j)| 0\le i <a, 0\le j <b\}$ is a complete set
of coset representatives in $\ZZ^2/\Delta$.
\end{corollary}

\begin{corollary}\label{number of halfes}
Given an integer $n$, by $S(n)$ denote the number of pairs
$(\Delta,\nu)$, where $\Delta$ varies over all subgroup of index $n$
in $\ZZ^2$ and $\nu$ is a coset of $\ZZ^2/\Delta$ with $2\nu=0$.
Then
$$
S(n)=\sigma_1(n)+3\sigma_1(\frac{n}{2}).
$$
\end{corollary}

For the proof see Corollary 1 in \cite{We2}.

We also need the following fact.

\begin{lemma}\label{lemG4-1.2}\label{index_and_ker}
Let $G$ be an abelian group and $H$ its subgroup of finite index.
Let $\phi: G \to G$ be an endomorphism of $G$, such that
$\phi(H)\leqslant H$ and the index $|G:\phi(G)|$ is also finite.
Then the cardinality of kernel of $\phi: G/H \to G/H$ is equal to
the index $|G : (H+\phi(G))|$.
\end{lemma}
\begin{proof}
Indeed,
$$
|\ker_{\phi}(G/H)|=\frac{|G/H|}{|\phi(G)/(\phi(G)\bigcap H)|}
$$

By the Second Isomorphism Theorem $\phi(G)/(\phi(G)\cap H)\cong
(\phi(G)+H)/H$. So
$$
|\ker_{\phi}(G/H)|=\frac{|G/H|}{(\phi(G)+H)/H}=|G/(\phi(G)+H)|.
$$
\end{proof}

\begin{remark}\label{number of halfes-remark}
Combining \Cref{index_and_ker} and \Cref{number of halfes} we get
the following observation. Consider a finite index subgroup
$H\leqslant \ZZ^2$. The number of $\nu \in \ZZ^2/H$ such that
$2\nu=0$ is equal to $|\ZZ^2/\langle (2,0), (0,2), H\rangle|$. This
can be proved by application of \Cref{index_and_ker} to the subgroup
$H$ and the endomorphism $\phi: g\to 2g,\, g\in \ZZ^2$. Since for
each $H$ the numbers $|\{\nu| \nu\in \ZZ^2/H,\,2\nu=0\}|$ and
$|\ZZ^2/\langle (2,0), (0,2), H\rangle|$ coincide, their sums taken
over all subgroups $H$ also coincide, that is
$$
S(n)=\sum_{H\leqslant \ZZ^2,\,|\ZZ^2/H|=n}|\{\nu| \nu\in
\ZZ^2/H,\,2\nu=0\}|=\sum_{H\leqslant \ZZ^2,\,|\ZZ^2/H|=n}
|\ZZ^2/\langle (2,0), (0,2), H\rangle|.
$$
\end{remark}

\begin{corollary}\label{number of mirror-preserved}
Let $\;\ell: \ZZ^3 \mapsto \ZZ^3$ be an automorphism of $\ZZ^3$,
given by $\ell(x,y,z)=(-x,y,z)$. Then the number of subgroups
$\Delta < \ZZ^3, \; |\ZZ^3:\Delta|=n$ with $\ell(\Delta)=\Delta$ is
$\sigma_2(n)+3\sigma_2(\frac{n}{2})$.
\end{corollary}

\begin{proof}
By \Cref{enumeration_of sublattices}, any subgroup $\Delta$ of
finite index in $\ZZ^3$ is generated by the rows of the corresponding matrix $\begin{pmatrix} c & e & f \\
0 & b & d \\ 0 & 0 & a\end{pmatrix}$. So the condition
$\ell(\Delta)=\Delta$ is equivalent to $\langle(0,0,a), (0,b,d),
(c,e,f)\rangle = \langle(0,0,a), (0,b,d), (-c,e,f)\rangle$.  Replace
$(c,e,f)$ with $(-c,-e,-f)$. Obviously, this does not change the
group. It follows that $(-c,-e,-f) \equiv (-c,e,f) \mod
\langle(0,0,a), (0,b,d)\rangle$.  The latter implies
$(0,2e,2f)\equiv 0 \mod \langle(0,0,a), (0,b,d)\rangle$. Denote
$H=\langle(0,0,a), (0,b,d)\rangle$. Fix the index $k=|\ZZ^2:H|$. By
\Cref{number of halfes} the number of such pairs $(H,(0,e,f))$
equals $\sigma_1(k)+3\sigma_1(\frac{k}{2})$. Summing over all the
possible values $k \mid n$ get the result.
\end{proof}

\begin{definition}
Let $G$ and $H$ be some groups, $\phi$ $\psi$ be automorphism of $G$
and $H$ respectively. We call the pairs $(G,\phi)$ and $(H,\psi)$
isomorphic if there exists an isomorphism $\xi: G \mapsto H$ such
that $\xi \circ \phi = \psi \circ \xi$.
\end{definition}

\begin{definition}
Let $\phi$ and $\psi$ be automorphisms of $\ZZ^2$. By
$f_{\ZZ^2,\phi}(n)$ we denote the number of subgroups $H < \ZZ^2$
such that $|\ZZ^2:H|=n$ and $\phi(H)\leqslant H$. Similarly by
$f_{\ZZ^2,\phi,\psi}(n)$ we denote the number of subgroups $H <
\ZZ^2$ such that $|\ZZ^2:H|=n$, $\phi(H)\leqslant H$ and the pair of
$H$ and the restriction of $\phi$ to $H$ is isomorphic to
$(H,\psi)$.
\end{definition}

The following automorphisms will be of the most impotence throughout
the article.

{\bf Notation.} By $\ell$ and $j$ denote the automorphisms of
$\ZZ^2$, given by $\ell: (u,v)\mapsto (u,-v)$ and $j: (u,v)\mapsto
(u,u-v)$ respectively.


\begin{proposition}\label{enumeration_of sublattices_with_automorphism_type_in_ZZ2}
The following identities hold
$$
f_{\ZZ^2,\ell,\ell}(n)=\sigma_0(n),\quad\quad\quad\quad\quad
f_{\ZZ^2,\ell,j}(n)=\sigma_0(\frac{n}{2}),\quad\quad\quad\quad\quad
f_{\ZZ^2,\ell}(n)=\sigma_0(n)+\sigma_0(\frac{n}{2});
$$
$$
f_{\ZZ^2,j,\ell}(n)=\sigma_0(\frac{n}{2}),\;\;
f_{\ZZ^2,j,j}(n)=\sigma_0(n)-2\sigma_0(\frac{n}{2})+2\sigma_0(\frac{n}{4}),\;\;
f_{\ZZ^2,j}(n)=\sigma_0(n)-\sigma_0(\frac{n}{2})+2\sigma_0(\frac{n}{4}).
$$
\end{proposition}

\begin{proof}

Given a subgroup $\Delta$ of index $n$ in $\ZZ^2$ consider its
corresponding matrix $\begin{pmatrix}   b & d \\
 0 & a \end{pmatrix}$ as described in \Cref{enumeration_of sublattices}.
Denote $X_\Delta=(0,a)$ and $Y_\Delta=(b,d)$.


Assume a subgroup $\Delta$ is preserved by $\ell$. Then $\langle
(0,a),(b,d) \rangle=\langle (0,-a),(b,-d)\rangle$, or equivalently
$2d \equiv 0 \mod a$. Recalling that $0 \le d < a$, we have $2d=0$
or $2d=a$, the latter is possible only in the case of even $a$. In
the first case $\ell(X_\Delta)=-X_\Delta$ and
$\ell(Y_\Delta)=Y_\Delta$, so the automorphism $\phi: \ZZ \to
\Delta$ given by $\phi((1,0))=Y_\Delta$ and $\phi((0,1))=X_\Delta$
provides an isomorphism $(\Delta,\ell)\cong (\ZZ^2,\ell)$.

In case $2d=a$ we have $\ell(X_\Delta)=-X_\Delta$ and
$\ell(Y_\Delta)=(b,-\frac{a}{2})=Y_\Delta-X_\Delta$. So
$\ell(Y_\Delta-X_\Delta)=Y_\Delta=(Y_\Delta-X_\Delta)+X_\Delta$.
Note that elements $X_\Delta$ and $Y_\Delta-X_\Delta$ generate the
same group $\langle X_\Delta, Y_\Delta-X_\Delta \rangle=\langle
X_\Delta, Y_\Delta\rangle=\Delta$. So an automorphism $\phi: \ZZ \to
\Delta$ given by $\phi((1,0))=Y_\Delta-X_\Delta$ and
$\phi((0,1))=X_\Delta$ provides an isomorphism $(\Delta,\ell)\cong
(\ZZ^2,j)$.


So each factorization $ab=n$ in the case of an odd $a$ provides one
$\ell$-invariant subgroup $\Delta$, moreover $(\Delta,\ell)\cong
(\ZZ^2,\ell)$. In the case of an even $a$ we get two subgroups
$(\Delta_1,\ell)\cong (\ZZ^2,\ell)$ and $(\Delta_2,\ell)\cong
(\ZZ^2,j)$. That is
$$
f_{\ZZ^2,\ell,\ell}(n)=\sum_{a\mid n}1=\sigma_0(n),\quad
f_{\ZZ^2,\ell,j}(n)=\sum_{a\mid n,\,2\mid
a}1=\sigma_0(\frac{n}{2}),\quad
f_{\ZZ^2,\ell}(n)=f_{\ZZ^2,\ell,\ell}(n)+f_{\ZZ^2,\ell,j}(n).
$$



The enumeration of $j$-invariant subgroups $\Delta$ follows the
similar way. If $j(\Delta)=\Delta$ then $\langle (0,a),(b,d)
\rangle=\langle (0,-a),(b,b-d)\rangle$, or equivalently $b-2d \equiv
0 \mod a$. Define integers $m$ and $k$ by $b-2d=ma$ and $2k=m$ or
$2k+1=m$ depending on the parity of $m$. Then
$j(X_\Delta)=-X_\Delta$ and $j(Y_\Delta)=Y_\Delta+mX_\Delta$. So
$j(Y_\Delta+kX_\Delta)=(Y_\Delta+kX_\Delta)+(m-2k)X_\Delta$. Also
note that elements $X_\Delta$ and $Y_\Delta+kX_\Delta$ generate the
same group $\langle X_\Delta, Y_\Delta+kX_\Delta \rangle=\langle
X_\Delta, Y_\Delta\rangle=\Delta$. So the automorphism $\phi: \ZZ
\to \Delta$ given by $\phi((1,0))=Y_\Delta+kX_\Delta$ and
$\phi((0,1))=X_\Delta$ provides an isomorphism  $(\Delta,j)\cong
(\ZZ^2,\ell)$ or $(\Delta,j)\cong (\ZZ^2,j)$ in cases $m$ is even or
odd respectively.

Now for each factorization $ab=n$ we find out, what solutions the
equation $b-2d=ma$ have under the restriction $0\le d <a$. In case
$a$ is odd there are a unique solution for each factorization,
moreover the parities of $b$ and $m$ coincide. In case $a$ is even
and $b$ is odd there are no solutions. In case $a$ and $b$ are even
there are two solutions for each factorization, one with even and
one with odd $m$. That is
$$
f_{\ZZ^2,j,j}(n)=\sum_{a\mid n,\,2\nmid a,\, 2\nmid
\frac{n}{a}}1+\sum_{a\mid n,\,2\mid a,\, 2\mid
\frac{n}{a}}1=\sigma_0(n)-2\sigma_0\big(\frac{n}{2}\big)+2\sigma_0\big(\frac{n}{4}\big);
$$
$$
f_{\ZZ^2,j,\ell}(n)=\sum_{a\mid n,\,2\mid
\frac{n}{a}}1=\sigma_0\big(\frac{n}{2}\big);\quad\quad
f_{\ZZ^2,j}(n)=f_{\ZZ^2,j,\ell}(n)+f_{\ZZ^2,j,j}(n).
$$
\end{proof}


{\bf Observation.} Indeed, it was proven above that for an arbitrary
involutory isomorphism $f: \ZZ^2 \to \ZZ^2$ the pair $(\ZZ^2,f)$ is
isomorphic to one of the following four: $(\ZZ^2,id)$,
$(\ZZ^2,-id)$, $(\ZZ^2,\ell)$ or $(\ZZ^2,j)$. But we do not use this
fact further.

\subsection{The structure of subgroups in the fundamental group of Klein bottle.}

The purpose of this chapter is to provide the enumeration of
subgroups and the conjugacy classes of subgroups in the  fundamental
group of Klein bottle.

{\bf Notation.} By $\Gamma$ denote the group, generated by $x,y$
with the relation $yxy^{-1}=x^{-1}$.  Note that $\Gamma$ is
isomorphic to the fundamental group of Klein bottle (see, for
example, \cite{Hatch}, p.72).

\begin{lemma}\label{structure of K}
\begin{itemize}
\item[(i)] Each element  $g\in\Gamma$ can be represented in the canonical form $g=x^ay^b$
for some integer $a,b$.
\item[(ii)] The product of two canonical forms is given by the
formula
\begin{equation}\label{multlawGamma}
x^ay^b \cdot x^{c}y^{d}= \left\{
\begin{aligned}
x^{a+c}y^{b+d} \quad \text{if} \quad b\equiv 0\mod2 \\
{x}^{a-c}{y}^{b+d} \quad \text{if} \quad b\equiv 1\mod2, \\
\end{aligned} \right.
\end{equation}
or shortly $x^ay^b \cdot x^{c}y^{d}=x^{a+(-1)^bc}{y}^{b+d}$.
\item[(iii)] The representation in the canonical form $g=x^ay^b$ for each element $g \in \Gamma$ is
unique.
\end{itemize}
\end{lemma}

\begin{proof}
Parts (i) and (ii) follows routinely from
the definition of the group, thus we concentrate on (iii). First,
note that the subgroup, generated by $x$ is normal in $\Gamma$.
Indeed, $x^y=x^{-1}$ -- this is exactly the relation of the group.
So $x^g \in \{x,x^{-1}\}$ for any $g \in \langle x,y \rangle =
\Gamma$. This means that the factorization $\phi: \Gamma\mapsto
\Gamma/\langle x \rangle$ is well-defined. Obviously, $\phi(y)$
generates $\phi(\Gamma)\cong\ZZ$.

Secondly, note that $x$ have the infinite order in $\Gamma$. This
follows from the Magnus Theorem (see \cite{Mag}). Indeed, the
relation $yxy^{-1}x$ is cyclically reduced and contains $x$, thus
$x$ generates a free subgroup. For more direct proof see
(\cite{Farkas}, Th. 10).

Assume the canonical representation of some element is not unique,
that is $x^ay^b=x^cy^d$ holds for some $(a,b)\neq (c,d)$. Applying
$\phi$ to both parts one gets $b=d$, consequently $a \neq c$. But
$x^ay^b=x^cy^d$ implies $x^a=x^c$.  This contradicts the infinite
order of $x$.\end{proof}

\begin{proposition}\label{enumeration_of subbottles}
The subgroups of index $n$ in $\Gamma$ are in one-to-one correspondence with the matrices $\begin{pmatrix}   b & d \\
 0 & a \end{pmatrix}$, where $ab=n$, $0 \le d < a$. A subgroup $\Delta$ of index $n$ is generated by
elements $x^a$ and $x^dy^b$ where $a,b,d$ are elements of the
corresponding matrix. Finally, $\Delta\cong\ZZ^2$ if $b$ is even and
$\Delta\cong \Gamma$ if $b$ is odd.
\end{proposition}

\begin{proof}
We consider all elements of $\Gamma$ to be represented in the
canonical form, provided by \Cref{structure of K}. To build the map
from subgroups $\Delta$ of index $n$ to matrices of prescribed form
do the following. Consider the minimal positive integer $a$, such
that $X_{\Delta}=x^a\in \Delta$. Such $a$ exists since index of
$\Delta$ in $\Gamma$ is finite.

Now consider the minimal positive integer $b$, for which there
exists integer $s$ such that the element $x^sy^b\in \Delta$ belongs
to $\Delta$. Put $Y_{\Delta}=X_{\Delta}^{-[s/a]}x^sy^b=x^dy^b$,
where $[r]$ denotes the maximal integer $t$ such that $t \le r$.
Note that $Y_{\Delta}\in \Delta$ and $d$ satisfy $0 \le d < a$.
Correspondence part is done.

To show that different matrices provide different subgroups $\Delta$
consider the following sets:
$$
\{x^{ia+jd}y^{jb}|\,i,j \in \ZZ\}.
$$
in case of even $b$ and
$$
\{x^{ia}y^{jb}|\,i,j \in \ZZ,\,2\mid
j\}\bigcup\{x^{ia+d}y^{jb}|\,i,j \in \ZZ,\,2\nmid j\}.
$$ in case of odd $b$. Direct verification through equation \ref{multlawGamma}
shows that the above sets are subgroups in $\Gamma$. They are
obviously generated by $X_{\Delta}, Y_{\Delta}$.

Now we prove the isomorphism part, we do it separately for different
parities of $b$. Fix a subgroup $\Delta$, consider corresponding
matrix, suppose $b$ is odd. Denote $X=x^a$ and $Y=x^dy^b$. Note that
$X$ and $Y$ satisfy the relation $YXY^{-1}=X^{-1}$. Thus the map $x
\to X, \,\, y \to Y$ can be extended to an epimorphism $\Gamma \to
\Delta $, so we only have to prove this epimorphism is an
isomorphism. As it is shown in the proof of \Cref{structure of K},
by means of relation $YXY^{-1}=X^{-1}$ one can reduce any element of
$\langle X,Y\rangle$ to $X^sY^t$ for some integers $s,t$. We have to
prove that this representation is unique. Suppose $X^sY^t=x^uy^v$,
using \Cref{multlawGamma} and definition of $X,Y$ one can show
$$
X^sY^t=\left\{
\begin{aligned}
x^{sa}y^{tb} \quad \text{if} \quad t\equiv 0\mod2 \\
{x}^{sa+d}{y}^{tb} \quad \text{if} \quad t\equiv 1\mod2 \\
\end{aligned} \right.
$$
or
$$
(u,v)=\left\{\begin{aligned}
(sa,tb) \quad \text{if} \quad t\equiv 0\mod2 \\
(sa+d,tb) \quad \text{if} \quad t\equiv 1\mod2 \\
\end{aligned} \right.
$$

All we need from the previous formula is to show that at most one
pair $(s,t)$ corresponds to a pair $(u,v)$, and the pair $(u,v)$ is
unique in virtue of \Cref{structure of K}. Thus each element $g \in
\langle X,Y\rangle$ can be uniquely represented in the form
$g=X^sY^t$, so $x \to X, \,\, y \to Y$ spawns an isomorphism $\Gamma
\mapsto \langle X,Y\rangle$.

The case of even $b$ is similar. \end{proof}

\begin{remark}\label{all abelian}
Let $\Gamma_+=\langle x,y^2\rangle$ be the subgroup of index 2 in
$\Gamma$. It follows from the above consideration that each abelian
subgroup of finite index in $\Gamma$ lies in $\Gamma_+$.
\end{remark}

%


\begin{proof}[Proof of \Cref{enum_all for K}]
By \Cref{enumeration_of subbottles} a subgroup $\Delta\cong\ZZ^2$ of
index $n$ is given by a matrix $\begin{pmatrix}   b & d \\
 0 & a \end{pmatrix}$, where $ab=n$, $b$ is even and $0 \le d < a$.
Then
$$
s_{\ZZ^2,\pi_{1}(\mathcal{K})}(n)=\sum_{ab=n,\, 2\mid
b}a=\sigma_1(\frac{n}{2}).
$$
To enumerate conjugacy classes we identify subgroups with their
corresponding matrix, and consider how the conjugation changes the
corresponding matrix. Obviously, $a$ and $b$ are invariant. Also,
$Ad_x$ preserves $d$ and $Ad_y$ maps $d \mapsto a-d$. That is for a
fixed factorization $ab=n$ there are $\frac{a+1}{2}$ or
$\frac{a+2}{2}$ conjugacy classes in case of an odd or even $a$
respectively. Thus
$$
c_{\ZZ^2,\pi_{1}(\mathcal{K})}(n)=\sum_{ab=n,\, 2\mid b}\left\{
\begin{aligned}
\frac{a+1}{2} \; \text{if} \; 2\nmid a\\
\frac{a+2}{2} \; \text{if} \; 2\mid a
\end{aligned}\right.=\frac{1}{2}\Big(\sigma_1(\frac{n}{2})+\sigma_0(\frac{n}{2})+\sigma_0(\frac{n}{4})\Big).
$$
In case $\Delta\cong\pi_{1}(\mathcal{K})$ arguing similarly we get
$$
s_{\pi_{1}(\mathcal{K}),\pi_{1}(\mathcal{K})}(n)=\sum_{ab=n,\,
2\nmid b}a=\sigma_1(n)-\sigma_1(\frac{n}{2}).
$$
Also, $Ad_x$ maps $d \mapsto d+2$ (keep in mind that $d$ is a
residue modulo $a$) and $Ad_y$  maps $d \mapsto a-d$. That is for a
fixed factorization $ab=n$ there are $1$ or $2$ conjugacy classes in
case of an odd or even $a$ respectively. Thus
$$
c_{\pi_{1}(\mathcal{K}),\pi_{1}(\mathcal{K})}(n)=\sum_{ab=n,\,
2\nmid b}\left\{
\begin{aligned}
1 \; \text{if} \; 2\nmid a\\
2 \; \text{if} \; 2\mid a
\end{aligned}\right.=\sigma_0(n)-\sigma_0(\frac{n}{4}).
$$
\end{proof}

\section{The structure of groups  $\pi_1(\mathcal{B}_{3})$ and $\pi_1(\mathcal{B}_{4})$}
The following two propositions provides the canonical form of an
element in $\pi_{1}(\mathcal{B}_{3})=\langle x, y, z:
yxy^{-1}=zxz^{-1}=x^{-1},
    zyz^{-1}=y^{-1}
 \rangle$ and $\pi_{1}(\mathcal{B}_{4})=\langle x, y, z: yxy^{-1}=zxz^{-1}=x^{-1},
    zyz^{-1}=xy^{-1}
 \rangle$, they are counterparts of \Cref{structure of K}. As
 before, $\Gamma=\langle x,y: yxy^{-1}=x^{-1} \rangle$.

\begin{proposition}\label{propB3-1}
\begin{itemize}
\item[(i)] Each element $g\in\pi_{1}(\mathcal{B}_{3})$ can be represented in the canonical form $g=x^ay^bz^c$
for some integer $a,b,c$.
\item[(ii)] The product of two canonical forms is given by
\begin{equation}\label{multlawB3}
x^ay^bz^c \cdot x^{d}y^{e}z^{f}= \left\{
\begin{aligned}
x^{a+d}y^{b+e}z^{c+f} \quad \text{if} \quad b\equiv 0\mod2,\quad c\equiv 0\mod2\\
{x}^{a-d}{y}^{b+e}{z}^{c+f} \quad \text{if} \quad b\equiv 1\mod2,\quad c\equiv 0\mod2 \\
{x}^{a-d}{y}^{b-e}{z}^{c+f} \quad \text{if} \quad b\equiv 0\mod2,\quad c\equiv 1\mod2 \\
{x}^{a+d}{y}^{b-e}{z}^{c+f} \quad \text{if} \quad b\equiv 1\mod2,\quad c\equiv 1\mod2 \\
\end{aligned} \right.,
\end{equation}
or shortly
$$
x^ay^bz^c \cdot x^{d}y^{e}z^{f} =
x^{a+(-1)^{b+c}d}y^{b+(-1)^ce}z^{c+f}.
$$
\item[(iii)] The canonical epimorphism $\phi: \pi_{1}(\mathcal{B}_{3}) \to
\Gamma$ given by $\phi: x^ay^bz^c \to y^bz^c$ is well-defined.
\item[(iv)] The representation in the canonical form $g=x^ay^bz^c$ for each element $g \in \pi_{1}(\mathcal{B}_{3})$ is
unique.
\end{itemize}
\end{proposition}

To prove part (iv) we need the following lemma.

\begin{lemma}\label{torsion_free}
The groups $\pi_{1}(\mathcal{B}_{3})$ and $\pi_{1}(\mathcal{B}_{4})$
are torsion-free (that is, contain no elements of finite order).
\end{lemma}

\begin{proof} This is a particular case of the well-known statement that a finite
order isometry of Euclidean space has a fixed point, see, for
example, (\cite{Farkas} Th. 10). Since $\pi_{1}(\mathcal{B}_{3})$
and $\pi_{1}(\mathcal{B}_{4})$ are the fundamental groups of closed
Euclidean manifolds, thus they act as isometry groups on the common
universal covering $\mathbb{E}^3$; that is have no fixed
points.\end{proof}

\begin{proof}[Proof of \Cref{propB3-1}]
Items (i--iii) follows routinely from the representation
\ref{fund_B3} of the group $\pi_{1}(\mathcal{B}_{3})$. To prove (iv)
assume the opposite, that is $x^{a}y^{b}z^{c}=x^{a'}y^{b'}z^{c'}$
for some triples $(a,b,c)\neq (a',b',c')$. Applying $\phi$ to both
parts and using the uniqueness of the canonical form for $\Gamma$
(see \Cref{structure of K}) one gets $(b,c)=(b',c')$. In turn,
$a\neq a'$ and $x^{a}=x^{a'}$ is a contradiction with the infinite
order of $x$, provided by \Cref{torsion_free}.\end{proof}

Note that the group $\Gamma$ appearing both  in Propositions
\ref{propB3-1} and \ref{propB4-1} is the same group -- the
fundamental group of Klein bottle. This reflects the fact that both
manifolds  are circle bundles over the Klein bottle.

\begin{proposition}\label{propB4-1}
\begin{itemize}
\item[(i)] Each element $g\in\pi_{1}(\mathcal{B}_{4})$ can be represented in the canonical form $g=x^ay^bz^c$
for some integer $a,b,c$.
\item[(ii)] The product of two canonical forms is given by
\begin{equation}\label{multlawB4}
x^ay^bz^c \cdot x^{d}y^{e}z^{f}= \left\{
\begin{aligned}
x^{a+d}y^{b+e}z^{c+f} \quad \text{if} \quad b\equiv 0\mod2,\quad c\equiv 0\mod2\\
{x}^{a-d}{y}^{b+e}{z}^{c+f} \quad \text{if} \quad b\equiv 1\mod2,\quad c\equiv 0\mod2 \\
{x}^{a-d}{y}^{b-e}{z}^{c+f} \quad \text{if} \quad b,e\equiv 0\mod2,\quad c\equiv 1\mod2 \\
{x}^{a-d+1}{y}^{b-e}{z}^{c+f} \quad \text{if} \quad b\equiv 0\mod2,\quad c,e\equiv 1\mod2 \\
{x}^{a+d}{y}^{b-e}{z}^{c+f} \quad \text{if} \quad e\equiv 0\mod2,\quad b,c\equiv 1\mod2 \\
{x}^{a+d+1}{y}^{b-e}{z}^{c+f} \quad \text{if} \quad b\equiv 1\mod2,\quad c,e\equiv 1\mod2 \\
\end{aligned} \right.,
\end{equation}
Or shortly $$x^ay^bz^c \cdot
x^{d}y^{e}z^{f}=x^{a+(-1)^{b+c}+\frac{1-(-1)^{ce}}{2}}y^{b+(-1)^ce}z^{c+f}.
$$
\item[(iii)] The canonical epimorphism $\phi: \pi_{1}(\mathcal{B}_{4}) \to
\Gamma$, given by $\phi: x^ay^bz^c \to y^bz^c$ is well-defined.
\item[(iv)] The representation in the canonical form $g=x^ay^bz^c$ for each element $g \in \pi_{1}(\mathcal{B}_{4})$ is
unique.
\end{itemize}
\end{proposition}

Next proposition provides effective means to enumerate the finite
index subgroups of the group $\pi_{1}(\mathcal{B}_{3})$. It plays
the same role for the group $\pi_{1}(\mathcal{B}_{3})$ as
\Cref{enumeration_of sublattices} for $\ZZ^2$ and $\ZZ^3$ or
\Cref{enumeration_of subbottles} for $\Gamma$.

\begin{proposition}\label{enumeration_B_3}
The subgroups $\Delta$ of index $n$ in $\pi_{1}(\mathcal{B}_{3})$ are in one-to-one correspondence with the integer matrices $\begin{pmatrix}  c& e & f\\0& b & d \\
 0& 0 & a \end{pmatrix}$ such that
\begin{itemize}
\item[(i)] $a,b,c > 0$ and $abc=n$;
\item[(ii)] $0 \le d,f < a$;
\item[(iii)] if case $b$ is odd, $e$ is even and $0 \le e < 2b$; if $b$ is
even then $0 \le e < b$;
\item[(iv)] if $b$ is even and $e$ is odd, then $2d \equiv 0 \mod
a$;
\item[(v)] if $b$ is odd and $c$ is even, then $2f \equiv 0 \mod
a$;
\item[(vi)] if both $b,c$ are odd, then $2f \equiv 2d \mod
a$.
\end{itemize}
Herewith the subgroup $\Delta$ is generated by elements $x^a$,
$x^dy^b$ and $x^fy^ez^c$, where $a,b,c,d,e,f$ are elements of the
corresponding matrix.
\end{proposition}

\begin{proof}
Below we consider all elements of the group
$\pi_{1}(\mathcal{B}_{3})$ to be represented in the canonical form
provided by \Cref{propB4-1}. To build the map from subgroups
$\Delta$ of index $n$ to matrices of prescribed form do the
following. Consider the minimal positive integer $a$, such that the
element $X_{\Delta}=x^a\in \Delta$. Such an exponent $a$ exists
since index of $\Delta$ in $\pi_{1}(\mathcal{B}_{3})$ is finite.

Consider the minimal positive integer $b$ such that there exists
integer $s$ such that the element $x^sy^b\in \Delta$ belongs to
$\Delta$. Put $Y_{\Delta}=X_{\Delta}^{-[s/a]}x^sy^b=x^dy^b$. Note
that $Y_{\Delta}\in \Delta$ and $d$ satisfy $0 \le d < a$.

Unfortunately our final step depends upon the parity of $b$. In case
$b$ is even consider the triple $(c,e,f)$ such that $c,e,f$ are
non-negative integers, $x^fy^ez^c \in \Delta$ and the triple
$(c,e,f)$ is lexicographically minimal among such triples. Note that
$e < b$, otherwise the element $Y_\Delta^{-1}x^fy^ez^c$ have same
$c$ and smaller nonnegative $e$; multiplying by a suitable degree of
$X_\Delta$ one can make the exponent at $x$ also nonnegative. The
matrix is built. Put $Z_\Delta=x^fy^ez^c$.

The case $b$ is odd is done in a similar way, with the sole
difference that we consider triples $(c,e,f)$ with even $e$, and use
the element $Y_\Delta^{-2}x^fy^ez^c$ to prove $e < 2b$.

To prove the remaining properties (iv--vi) we do the following. Note
that if an element $x^u$ belongs to $\Delta$, then $a\mid u$.
Otherwise the existence of element $X_\Delta^{[-u/a]}x^u\in \Delta$
contradicts the definition of $a$. In case $b,c$ are even and $e$ is
odd consider an element $Y_\Delta Z_\Delta
Y_\Delta^{-1}Z_\Delta^{-1}=x^{2d}\in \Delta$; thus $a\mid 2d$. In
case $b$ is even and $c,e$ are odd  consider an element $Y_\Delta
Z_\Delta Y_\Delta Z_\Delta^{-1}=x^{2d}\in \Delta$; then $a\mid 2d$.
Together this two cases cover the case $b$ is even $e$ is odd. In
case $b$ is odd and $c$ is even consider the element $Y_\Delta
Z_\Delta Y_\Delta^{-1}Z_\Delta^{-1}=x^{-2f}\in \Delta$; we have
$a\mid 2f$. In case both $b,c$ are odd $Y_\Delta Z_\Delta Y_\Delta
Z_\Delta^{-1}=x^{2d-2f}\in \Delta$; this implies $a\mid 2d-2f$.

To build a correspondence from matrices to subgroups in all cases do
the following: denote $X=x^a$, $Y=x^dy^b$, $Z=x^fy^ez^c$ and
consider the set
$$
\{X^uY^vZ^w|u,v,w \in \ZZ\}
$$
Direct calculus using (\ref{multlawB3}) shows that this set is
closed by multiplication, then it forms a subgroup of index $n$.
\end{proof}

Next proposition is an exact analog of the previous one for manifold
$\mathcal{B}_{4}$.

\begin{proposition}\label{enumeration_B_4}
The subgroups $\Delta$ of index $n$ in $\pi_{1}(\mathcal{B}_{4})$ are in one-to-one correspondence with the integer matrices $\begin{pmatrix}  c& e & f\\0& b & d \\
 0& 0 & a \end{pmatrix}$ such that
\begin{itemize}
\item[(i)] $a,b,c > 0$ and $abc=n$;
\item[(ii)] $0 \le d,f < a$;
\item[(iii)] if case $b$ is odd, $e$ is even and $0 \le e < 2b$; if $b$ is
even then $0 \le e < b$;
\item[(iv)] if $b$ is even and $e$ is odd, then $2d \equiv 0 \mod
a$;
\item[(v)] if $b$ is odd and $c$ is even, then $2f \equiv 0 \mod
a$;
\item[(vi)] if both $b,c$ are odd, then $2(d-f) \equiv 1 \mod
a$.
\end{itemize}
Herewith the subgroup $\Delta$ is generated by elements $x^a$,
$x^dy^b$ and $x^fy^ez^c$, where $a,b,c,d,e,f$ are elements of the
corresponding matrix.
\end{proposition}

{\bf Observation.} Note that the sole difference between
\Cref{enumeration_B_3} and \Cref{enumeration_B_4} is in the item
(vi). In particular, $2(d-f)\equiv 1 \mod a$ implies $a$ is odd.
Thus $n$ is necessarily odd in this case, because $n=abc$.

The proof literarily repeats the proof of \Cref{enumeration_B_3}, so
we omit it. Next proposition shows that the invariants introduced in
\Cref{enumeration_B_3} are sufficient to determine the the
isomorphism type of a subgroup.

\begin{proposition}\label{classification_B_3}
Let $\Delta$ be a subgroup of finite index in
$\pi_{1}(\mathcal{B}_{3})$ described by matrix $\begin{pmatrix}  c& e & f\\0& b & d \\
 0& 0 & a \end{pmatrix}$
 determined in \Cref{enumeration_B_3}. Then
 \begin{itemize}
\item[(1)] If $b,c,e$ are even then $\Delta
\cong \ZZ^3$.
\item[(2)] If $c$ is odd and  $b,e$ are even then $\Delta \cong
\pi_1(\mathcal{G}_2)$.
\item[(3)] If $e$ is odd  and $d\equiv 0 \mod a$ ($b$ is necessarily even in this case) then $\Delta \cong
\pi_1(\mathcal{B}_1)$.
\item[(4)] If $e$ is odd and $d\equiv \frac{a}{2} \mod a$ ($a$ and $b$ are necessarily even in this case) then $\Delta
\cong \pi_1(\mathcal{B}_2)$.
\item[(5)] If $c,e$ are even,
$b$ is odd and $f\equiv 0 \mod a$ then $\Delta \cong
\pi_1(\mathcal{B}_1)$.
\item[(6)] If $c,e$ are even,
$b$ is odd and $f\equiv \frac{a}{2} \mod a$ ($a$ is necessarily even
in this case) then $\Delta \cong \pi_1(\mathcal{B}_2)$.
\item[(7)] If $c,b$ are odd and $f-d\equiv 0 \mod a$ ($e$ is necessarily even in this case) then $\Delta \cong
\pi_1(\mathcal{B}_3)$.
\item[(8)] If $c,b$ are odd and $f-d\equiv \frac{a}{2} \mod a$ ($a$ and $e$ are necessarily even in this case) then $\Delta \cong
\pi_1(\mathcal{B}_4)$.
\end{itemize}
\end{proposition}

{\bf Remark.} Note that the case $b$ and $e$ are both odd
 is prohibited by \Cref{enumeration_B_3}.

First we need the following lemma.
\begin{lemma}\label{unique_induced_representation}
For a triple of integers $(p,q,r)$ there is at most one integer
triple $(s,t,u)$ such that
$x^py^qz^r=X_\Delta^sY_\Delta^tZ_\Delta^u$.
\end{lemma}
\begin{proof}Since among $X_\Delta,Y_\Delta,Z_\Delta$ only the latter one have a non-zero exponent at $z$ in the canonical representation,
$r$ at most uniquely determines $u$. The element
$x^py^qz^rZ_\Delta^{-u}$ have zero exponent at $z$, thus its
exponent at $y$ at most uniquely determines $t$, similar way for
$p$.
\end{proof}

\begin{proof}[Proof of \Cref{classification_B_3}]
Consider a subgroup $\Delta$ and define $X_\Delta, Y_\Delta,
Z_\Delta$ as in the proof of \Cref{enumeration_B_3}. We need the
following lemma.

By \Cref{enumeration_B_3} the subgroup $\Delta$ is generated by
$X_\Delta, Y_\Delta, Z_\Delta$.
We use the following
standard representations of crystallographical groups (see
\cite{Wolf} or \cite{Conway})
\begin{equation}\label{all_relations}
\begin{aligned}
\ZZ^3&=\langle x,y,z:
xyx^{-1}y^{-1}=xzx^{-1}z^{-1}=yzy^{-1}z^{-1}=1\rangle,\\
\pi_1(\mathcal{G}_2)&=\langle x,y,z: xyx^{-1}y^{-1}=1,
x^z=x^{-1},y^z=y^{-1}\rangle,\\
\pi_1(\mathcal{B}_1)&=\langle x,y,z:
xyx^{-1}y^{-1}=yzy^{-1}z^{-1}=1, x^{z}=x^{-1}\rangle,\\
\pi_1(\mathcal{B}_2)&=\langle x,y,z: xyx^{-1}y^{-1}=1,
x^z=x^{-1},y^z=xy\rangle,\\
\pi_1(\mathcal{B}_3)&=\langle x,y,z: x^y=x^{z}=x^{-1},
y^{z}=y^{-1}\rangle,\\
\pi_1(\mathcal{B}_4)&=\langle x,y,z: x^y=x^{z}=x^{-1},
y^{z}=xy^{-1}\rangle.
\end{aligned}
\end{equation}
Below $G$ is one of the groups from the list \ref{all_relations}.
Our immediate goal is to build the bijection $\xi$ from the set
$\{x,y,z\}$ of generators of $G$ to the set $\{X_\Delta, Y_\Delta,
Z_\Delta\}$ of generators of $\Delta$, such that the relations of
$G$ hold for their images $\xi(x),\xi(y),\xi(z)$. This means that
$\xi$ can be extended to an epimorphism $\xi: G \hookrightarrow
\Delta$.

In case $G$ is one of the groups $\ZZ^3, \pi_1(\mathcal{G}_2),
\pi_1(\mathcal{B}_3), \pi_1(\mathcal{B}_4)$, put
$\xi(x)=X_\Delta,\;\xi(y)=Y_\Delta,\;\xi(z)=Z_\Delta$. For the cases
$\pi_1(\mathcal{B}_1)$ and $\pi_1(\mathcal{B}_2)$ the following
makes the job:
\begin{itemize}
\item if $c$ is even and $e$ is odd then
$\xi(x)=X_\Delta,\;\xi(y)=Y_\Delta,\;\xi(z)=Z_\Delta$;
\item if both $c,e$ are odd then $\xi(x)=Y_\Delta,\;\xi(y)=X_\Delta,\;\xi(z)=Z_\Delta$;
\item finally, if $b$ is odd then
$\xi(x)=X_\Delta,\;\xi(y)=Z_\Delta,\;\xi(z)=Y_\Delta$.
\end{itemize}
Direct verification through (\ref{multlawB3}) shows that relations
(\ref{all_relations}) hold.

Now we need to prove that the epimorphism $\xi$ is really an
isomorphism. We claim that in each of the groups $\ZZ^3$,
$\pi_1(\mathcal{G}_2)$, $\pi_1(\mathcal{B}_1)$,
$\pi_1(\mathcal{B}_2)$, $\pi_1(\mathcal{B}_3)$,
$\pi_1(\mathcal{B}_4)$ any element $g$ can be uniquely represented
in the canonical form $g=x^sy^tz^u$. This is obvious for $\ZZ^3$;
for $\pi_1(\mathcal{G}_2)$ see (\cite{We2}, Proposition 2); for
$\pi_1(\mathcal{B}_1)$ and $\pi_1(\mathcal{B}_2)$ see (\cite{We1},
Propositions 1 and 6 respectively); for $\pi_1(\mathcal{B}_3)$ and
$\pi_1(\mathcal{B}_4)$ this is \Cref{propB3-1} and \Cref{propB4-1}
above. Thus each element $g\in \Delta$ can be represented in the
form $g=X_\Delta^sY_\Delta^tZ_\Delta^u$. If $\xi$ have non-trivial
core then some $g\in \Delta$ have two different representations.
Recall that $\Delta$ is a subgroup of $\pi_1(\mathcal{B}_3)$, then
by \Cref{propB3-1} any element $g$ can be represented in the form
$g=x^py^qz^r$. Thus the existence of two different representations
$g=X_\Delta^sY_\Delta^tZ_\Delta^u$ contradicts
\Cref{unique_induced_representation}.
\end{proof}

Next proposition is a twin of \Cref{classification_B_3} for the
group $\pi_{1}(\mathcal{B}_{4})$.

\begin{proposition}\label{classification_B_4}
Let $\Delta$ be a subgroup of finite index in
$\pi_{1}(\mathcal{B}_{4})$ described by matrix $\begin{pmatrix}  c& e & f\\0& b & d \\
 0& 0 & a \end{pmatrix}$
 determined in \Cref{enumeration_B_3}. Then
 \begin{itemize}
\item[(1)] If $b,c,e$ are even then $\Delta
\cong \ZZ^3$.
\item[(2)] If $c$ is odd and  $b,e$ are even then $\Delta \cong
\pi_1(\mathcal{G}_2)$.
\item[(3)] If $e$ is odd  and $d\equiv 0 \mod a$ ($b$ is necessarily even in this case) then $\Delta \cong
\pi_1(\mathcal{B}_1)$.
\item[(4)] If $e$ is odd and $d\equiv \frac{a}{2} \mod a$ ($a$ and $b$ are necessarily even in this case) then $\Delta
\cong \pi_1(\mathcal{B}_2)$.
\item[(5)] If $c,e$ are even,
$b$ is odd and $f\equiv 0 \mod a$ then $\Delta \cong
\pi_1(\mathcal{B}_1)$.
\item[(6)] If $c,e$ are even,
$b$ is odd and $f\equiv \frac{a}{2} \mod a$ ($a$ is necessarily even
in this case) then $\Delta \cong \pi_1(\mathcal{B}_2)$.
\item[(7)] If $c,b$ are odd (recall that in this case necessarily $a,e$ are odd and $2(d-f) \equiv 1 \mod
a$) then $\Delta \cong \pi_1(\mathcal{B}_4)$.
\end{itemize}
\end{proposition}

The proof repeats exactly the proof of \Cref{classification_B_3}, so
we omit it.

\section{Proof of \Cref{th-1-amphidi+} and
\Cref{th-2-amphidi+}\label{chapt_Amphi+}}
\subsection*{Overall scheme of the proof}\label{Overall_scheme}

In this chapter we prove the Theorems \ref{th-1-amphidi+} and
\ref{th-2-amphidi+} separately for each isomorphism
type.\par\smallskip

\noindent{\bf Notations.} By $\Lambda$ denote the subgroup $\Lambda=
\langle x, y^2, z^2 \rangle$ in the group
$\pi_1(\mathcal{B}_3)=\langle x, y, z \rangle$.\par\smallskip

Note that $\Lambda \cong \ZZ^3$, $\Lambda \lhd \pi_1(\mathcal{B}_3)$
and $\pi_1(\mathcal{B}_3)/\Lambda \cong \ZZ_2\oplus\ZZ_2$; the
cosets are represented by $1$, $y$, $z$ and $yz$.

\begin{definition}
Recall that $\Delta^\Lambda$ denotes the set of subgroups
$\{\Delta^\lambda|\; \lambda\in\Lambda\}$, further we call it {\em
an intermediate conjugacy class} of $\Delta$. Denote the number of
intermediate conjugacy classes of subgroups $\Delta$ of a given
isomorphism type $H$ by $c_{H,\pi_1(\mathcal{B}_3)}^\Lambda(n)$.
\end{definition}

Recall that $Ad_g$ provides an automorphism $u \mapsto u^g$ of the
group $\pi_1(\mathcal{B}_3)$.

\begin{definition}
Given an isomorphism type $H$ of a subgroup $\Delta$, consider
intermediate conjugacy classes $\Delta^\Lambda$. The number of
intermediate conjugacy classes $\Delta^\Lambda$ preserved by $Ad_1$,
$Ad_y$, $Ad_z$ and $Ad_{yz}$ denote
$N_{H,\pi_1(\mathcal{B}_3),1}(n)$,
$N_{H,\pi_1(\mathcal{B}_3),y}(n)$, $N_{H,\pi_1(\mathcal{B}_3),z}(n)$
and $N_{H,\pi_1(\mathcal{B}_3),yz}(n)$ respectively. In particular,
$N_{H,\pi_1(\mathcal{B}_3),1}(n)=c_{H,\pi_1(\mathcal{B}_3)}^\Lambda(n)$,
but we introduce this notation for uniformity. Further we omit $H$
and $\pi_1(\mathcal{B}_3)$ as indexes if the context is clear.
\end{definition}

Our calculation of $s_{H,\pi_1(\mathcal{B}_3)}(n)$ will de
straightforward, while the calculation of
$c_{H,\pi_1(\mathcal{B}_3)}(n)$ will take the several steps. First
we find $c_{H,\pi_1(\mathcal{B}_3)}^\Lambda(n)$. Essentially,
\Cref{enumeration_B_3} claims that a subgroup $\Delta$ is uniquely
determined by the choice of three items. The first is a
factorization $abc=n$, the second is a coset of the element $x^d$ in
$\langle x \rangle/\langle x^a \rangle$ (which we simplified by the
restriction $0 \le d < a$) and the third is a coset of $x^fy^e$ in
$\langle x,y \rangle/\langle x^a,x^dy^b \rangle$ (similarly, this
was simplified by restricting $0\le e < b$ and $0 \le f <a$). The
conjugation by generators of $\Lambda$ has the following properties:
\begin{itemize}
\item the replacement $\Delta \mapsto \Delta^x$  maps $x^d \mapsto x^d[x,y^b]$ and $x^fy^e  \mapsto x^f[x,y^ez^c]y^e
$; where we denote $[u,v]=uvu^{-1}v^{-1}$
\item the replacement $\Delta \mapsto \Delta^{y^2}$ maps  $x^fy^e  \mapsto x^fy^e
[y^2,z^c]$;
\item the replacement $\Delta \mapsto \Delta^{z^2}$ is always identity transformation.
\end{itemize}
This allows to calculate $c_{H,\pi_1(\mathcal{B}_3)}^\Lambda(n)$
uniformly for different isomorphism types $H$. Calculations of
$N_1(n)$,  $N_y(n)$, $N_z(n)$ and $N_{yz}(n)$ use \Cref{number of
mirror-preserved} and similar considerations. An application of the
Burnside Lemma finishes the job.

\subsection{Case $\Delta \cong \pi_{1}(\mathcal{G}_{1})\cong \ZZ^3$\label{ZinB3}}
By \Cref{enumeration_B_3} and \Cref{classification_B_3} every
subgroup $\Delta\cong \ZZ^3$ of index $n$ in
$\pi_{1}(\mathcal{B}_{3})$ corresponds to an integer matrix $\begin{pmatrix}  c& e & f\\0& b & d \\
 0& 0 & a \end{pmatrix}$, where $a,b,c>0$, $abc=n$; $b, c, e$ are even and $0\le d,f
<a$, $0 \le e <b$. Thus
$$
\begin{aligned}
s_{G_1,B_3}(n)&=|\{(a,b,c,d,e,f)\in \ZZ^6|a,b,c>0,\,abc=n,\,2\mid
b,c,e,\,0\le d,f <a,\,0\le e
<b\}|\\&=\sum_{abc=\frac{n}{4}}a^2b=\omega(\frac{n}{4}).
\end{aligned}
$$

Note that $\Delta\cong \ZZ^3$ implies $\Delta \leqslant \Lambda$;
recall that $\Lambda \cong \ZZ^3$, therefore
$\Delta^\Lambda=\{\Delta\}$. Indeed, this is another way to get the
equality $s_{G_1,B_3}(n)=\omega(\frac{n}{4})$. So,
$c_{G_1,B_3}^\Lambda(n)=N_1(n)=\omega(\frac{n}{4})$. Since
$\Lambda=\langle x, y^2, z^2\rangle\cong\ZZ^3$, we use the additive
notation for $\Lambda$, representing $x$, $y^2$ and $z^2$ by
$(1,0,0)$, $(0,1,0)$ and $(0,0,1)$ respectively. Then $Ad_y$, $Ad_z$
and $Ad_{yz}$ are given by $Ad_y: (u,v,w)\mapsto (-u,v,w)$, $Ad_z:
(u,v,w)\mapsto (-u,-v,w)$ and $Ad_{yz}: (u,v,w)\mapsto (u,-v,w)$.
Thus by \Cref{number of mirror-preserved},
$N_y(n)=N_z(n)=N_{yz}(n)=\sigma_2(\frac{n}{4})+3\sigma_2(\frac{n}{8})$.
Applying the Burnside Lemma one gets
$$
c_{G_1,B_3}(n)=\frac{1}{4}\Big(\omega(\frac{n}{4})+3\sigma_2(\frac{n}{4})+9\sigma_2(\frac{n}{8})\Big).
$$

\subsection{Case $\Delta \cong \pi_{1}(\mathcal{G}_{2})$\label{G2inB3}}
By \Cref{enumeration_B_3} and \Cref{classification_B_3}, every
subgroup $\Delta\cong \pi_{1}(\mathcal{G}_{2})$ of index $n$ in
$\pi_{1}(\mathcal{B}_{3})$ corresponds to an integer matrix $\begin{pmatrix}  c& e & f\\0& b & d \\
 0& 0 & a \end{pmatrix}$, where $a,b,c>0$, $abc=n$; $b,e$ are even, $c$ is odd and $0\le d,f
<a$, $0 \le e <b$. Thus
$$
\begin{aligned}
s_{G_1,B_3}(n)&=|\{(a,b,c,d,e,f)\in \ZZ^6|a,b,c>0,\,abc=n,\,2\mid
b,e,\,2\nmid c,\,0\le d,f <a,\,0\le e
<b\}|\\&=\sum_{\substack{abc=n,\\2\mid b,\, 2\nmid
c}}a^2\frac{b}{2}=\sum_{abc=\frac{n}{2}}a^2b-\sum_{abc=\frac{n}{4}}a^2b=\omega(\frac{n}{2})-\omega(\frac{n}{4}).
\end{aligned}
$$
So for the case $\Delta \cong \pi_{1}(\mathcal{G}_{2})$ we are done
with \Cref{th-1-amphidi+} in this case, and proceed to
\Cref{th-2-amphidi+}.

{\bf Notations.} Recall that $X_\Delta=x^a$, $Y_\Delta=x^dy^b$ and
$Z_\Delta=x^fy^ez^c$. Put $H_\Delta=\langle X_\Delta,
Y_\Delta\rangle$ and $J_H=\langle x, y^2 \rangle/\langle x^2, y^4,
H_\Delta \rangle$.

Observe that $|J_H|\in \{1,2,4\}$ for every $H$.

To enumerate intermediate conjugacy classes note that
$X_\Delta=X_\Delta^\lambda$ and $Y_\Delta=Y_\Delta^\lambda$ for all
$\lambda \in \Lambda$. Also $[Z_\Delta,x]=x^2$, $[Z_\Delta,y^2]=y^4$
and $[Z_\Delta,z^2]=1$. Since $\Lambda=\langle x, y^2, z^2\rangle$,
the family of intermediate conjugacy classes
$\{\Delta^\Lambda|H\cong\pi_{1}(\mathcal{G}_{2}),|
\pi_{1}(\mathcal{B}_{3}):H|=n\}$ is enumerated by a choice of the
following invariants:
\begin{itemize}
\item a factorization $abc=n$, where $b$ is even and $c$ is odd,
\item a subgroup $H < \langle x,
y^2\rangle$ such that $|\langle x, y\rangle:H|=ab$,
\item a coset of the element $Z_\Delta z^{-c}$ in $J_H$.
\end{itemize}

So we have to enumerate the sequences of choices. Note that the
number of sequences for a fixed $c$ is provided by \Cref{number of
halfes-remark}, namely the number of pairs $(H,h)$ such that
$|\ZZ^2:H|=\frac{n}{2c},\, h\in I_H$ is equal to
$\sigma_1(\frac{n}{2c})+3\sigma_1(\frac{n}{4c})$. Summing over all
odd values of $c$ one gets
$$
c_{G_2,B_3}^\Lambda(n)=\sigma_2(\frac{n}{2})+2\sigma_2(\frac{n}{4})-3\sigma_2(\frac{n}{8}).
$$

Now we find $N_y(n)$, $N_z(n)$ and $N_{yz}(n)$. Note that $Ad_z$
preserves any $\Delta^\Lambda$, in other words $N_z(n)=N_1(n)$ and
$N_y(n)=N_{yz}(n)$.

From this point until the end of the subsection $H$ always denotes a
subgroup of $\langle x, y^2\rangle \cong \ZZ^2$. For the latter we
use the additive notation, representing $x$ and $y^2$  by $(1,0)$
and $(0,1)$ respectively.

{\bf Notation.} There are 4 cosets in the coset decomposition
$\ZZ^2/(2\ZZ)^2$, represented by elements $(0,0)$, $(1,0)$, $(0,1)$
and $(1,1)$ respectively. We denote this cosets by
$\mathcal{M}_{0,0}$, $\mathcal{M}_{1,0}$, $\mathcal{M}_{0,1}$,
$\mathcal{M}_{1,1}$. We define the type of a subgroup $H < \ZZ^2$ as
follows:
\begin{itemize}
\item $H$ have the type $(0,0)$ if $H \subseteq \mathcal{M}_{0,0}$
\item $H$ have the type $(1,0)$ if $H \subseteq
\mathcal{M}_{0,0}\bigcup\mathcal{M}_{1,0}$ but $H \not\subseteq
\mathcal{M}_{0,0}$
\item $H$ have the type $(0,1)$ if $H \subseteq
\mathcal{M}_{0,0}\bigcup\mathcal{M}_{0,1}$ but $H \not\subseteq
\mathcal{M}_{0,0}$
\item $H$ have the type $(1,1)$ if $H \subseteq
\mathcal{M}_{0,0}\bigcup\mathcal{M}_{1,1}$ but $H \not\subseteq
\mathcal{M}_{0,0}$
\item finally, $H$ is of total type  if none of the previous cases
hold
\end{itemize}

Note that $|J_H|$ depends only on the type of a subgroup $H$, and
equals 4,2,2,2,1 respectively.

Denote by $G_{0,0}(n)$, $G_{1,0}(n)$, $G_{0,1}(n)$, $G_{1,1}(n)$ and
$G_{total}(n)$ the number of subgroups $n$ of respective type having
index $n$ in $\ZZ^2$ and preserved by automorphism $(u,v)\mapsto
(u,-v)$.

To calculate $N_y(n)$ we need to enumerate pairs $(H,h)$ of
described above form, such that the composition of correspondences
$(H,h) \to \Delta \to \Delta^y \to (H',h')$ acts identically on
them. Recalling all definitions one gets that $H'=H^y$ and
$h'=h[Z_\Delta,y]$. Note that $[Z_\Delta,y]$ does not depends on
$\Delta$, moreover $[Z_\Delta,y]=[z,y]=y^2$. So the condition $h'=h$
is equivalent to $y^2\in J_H$, that is holds true for all $h\in J_H$
if $H$ have type $(0,1)$ or total, and for none $h\in J_H$ if $H$
have any other type. Recall that $|J_H|$ is equal to $2$ and $1$ is
case $H$ have type $(0,1)$ and total respectively. So we have to
find the sum of $2G_{0,1}(k)+G_{total}(k)$ over all possible
$k=|\langle x, y^2\rangle:H|$.

Fix an index $k=|\ZZ^2:H|$. Note that $Ad_y$ have the form
$(u,v)\mapsto (u,-v)$ on $\langle x, y^2\rangle = \langle
(1,0),(0,1)\rangle$. So we claim the following isomorphisms:
\begin{align*}
(\langle(1,0)\!,\!(0,1)\rangle\!,\!Ad_y)\!&\cong\!(\langle
(2,0)\!,\!(0,1)\rangle\!,\!Ad_y)\!\cong\!(\langle
(1,0)\!,\!(0,2)\rangle\!,\!Ad_y)\!\cong\!(\langle
(2,0)\!,\!(0,2)\rangle\!,\!Ad_y)\!\cong\!(\ZZ^2,\ell),\\
(\langle(1,1)\!,\!(0,2)\rangle\!,\!Ad_y)\!&\cong\!(\ZZ^2,j).
\end{align*}
That is
$$
G_{0,0}(k)+G_{1,0}(k)+G_{0,1}(k)+G_{1,1}(k)+G_{total}(k)=f_{\ZZ^2,\ell}(k),
$$
$$
G_{0,0}(k)+G_{1,0}(k)=G_{0,0}(k)+G_{0,1}(k)=f_{\ZZ^2,\ell}(\frac{k}{2}),
$$
$$
G_{0,0}(k)+G_{1,1}(k)=f_{\ZZ^2,j}(\frac{k}{2}),
$$
$$
G_{0,0}(k)=f_{\ZZ^2,\ell}(\frac{k}{4}).
$$
Substituting values from \Cref{enumeration_of
sublattices_with_automorphism_type_in_ZZ2} we have
$2G_{0,1}(k)+G_{total}(k)=f_{\ZZ^2,\ell}(k)-f_{\ZZ^2,j}(\frac{k}{2})=\sigma_0(k)+d_2(\frac{k}{4})-2\sigma_0(\frac{k}{8})$.
Summing over all values $k$ such that $2k \mid n$ and $\frac{n}{2k}$
is odd, one gets
$N_y(n)=d_3(\frac{n}{2})-d_3(\frac{n}{4})+d_3(\frac{n}{8})-3d_3(\frac{n}{16})+2d_3(\frac{n}{32})$.

Finally,
$$
c_{G_2,B_3}=\frac{1}{2}\Big(\sigma_2(\frac{n}{2})+2\sigma_2(\frac{n}{4})-3\sigma_2(\frac{n}{8})+d_3(\frac{n}{2})-d_3(\frac{n}{4})+d_3(\frac{n}{8})-3d_3(\frac{n}{16})+2d_3(\frac{n}{32})\Big).
$$

\subsection{Case $\Delta \cong \pi_{1}(\mathcal{B}_{1})$\label{B1inB3}}

The difficulty of this case is that two types of matrices lead to
the same isomorphism type $\pi_{1}(\mathcal{B}_{1})$, see cases 3
and 5 of \Cref{classification_B_3}. We consider these cases
separately. 

\subsubsection{Case 3 of \Cref{classification_B_3}\label{case 3}}

In this case the corresponding matrix of a subgroup $\Delta$ has
the form $\begin{pmatrix}  c& e & f\\0& b & 0 \\
 0& 0 & a \end{pmatrix}$, where $b$ is even and $e$ is odd. So for a
 fixed triple $(a,b,c)$ there are $\frac{b}{2}$ choices for $e$ and
 $a$ choices for $f$. That is
$$
s_{B_1,B_3}^{case\,3}(n)=\sum_{\substack{abc=n,\\2\mid
b}}\frac{ab}{2} =\chi(\frac{n}{2}).
$$

Now we enumerate conjugacy classes of subgroups $\Delta$ matching
case 3. When replacing a subgroup $\Delta$ with a conjugated one
$\Delta^g, g\in \pi_{1}(\mathcal{B}_{3})$ the numbers $a,b,c$ are
not changed. The pair $(e,f)$ is changed in the following way.

First assume that $c$ is odd. Taking into account that
$\pi_{1}(\mathcal{B}_{3})=\langle x,y,z\rangle$ it is sufficient to
describe the action of $Ad_x$, $Ad_y$, $Ad_z$:
\begin{align*}
Ad_x: (e,f) \mapsto (e,f), && Ad_y: (e,f) \mapsto (e+2,-f), &&
Ad_z:(e,f) \mapsto (-e,-f).
\end{align*}
 Keep in mind that $e$ and
$f$ are residues modulo $b$ and $a$ respectively, for example, $-f$
refers to the number $a-f$. Also recall that $e$ is always odd; the
parity of $e$ is well-defined since $b$ is even.

So, in case $b \equiv 0 \mod 4$ there are $a$ orbits under the
action of the group generated by $Ad_x$, $Ad_y$ and $Ad_z$; each
residue $s$ modulo $a$ corresponds to the orbit $\{(4i+1,s)|0\le i
<\frac{b}{4}\}\cup\{(4i+3,-s)|0\le i <\frac{b}{4}\}$.

In case $b \equiv 2 \mod 4$ the orbits are of the form
$\{(2i+1,s)|0\le i <\frac{b}{2}\}\cup \{(2i+1,-s)|0\le i
<\frac{b}{2}\}$. Thus there are $\frac{a+1}{2}$ orbits if $a$ is
odd, and $\frac{a+2}{2}$ orbits if $a$ is even.

So, the number of conjugacy classes for a fixed odd $c$ is
$$
\begin{aligned}
\sum_{\substack{ab=\frac{n}{c}\\2\mid b}}\left\{
\begin{aligned}
\frac{a+1}{2}  \;&\text{if} \; 2 \nmid a,\, 4\nmid b\\
\frac{a+2}{2}  \;&\text{if} \; 2 \mid a,\, 4\nmid b\\
a\quad \;&\text{if} \quad 4\mid b
\end{aligned}\right.&=\frac{1}{2}\Bigg(\sum_{\substack{ab=\frac{n}{c}\\2\mid b}}\left\{
\begin{aligned}
a  \;\text{if} \;  4\nmid b\\
2a  \;\text{if} \;  4\mid b\\
\end{aligned}\right.+\sum_{\substack{ab=\frac{n}{c}\\2\mid b, 4\nmid b}}\left\{
\begin{aligned}
1  \;\text{if} \;  2\nmid a\\
2  \;\text{if} \;  2\mid a
\end{aligned}\right.\Bigg)\\&=
\frac{1}{2}\Big(\sigma_1\big(\frac{n}{2c}\big)+\sigma_1\big(\frac{n}{4c}\big)+\sigma_0\big(\frac{n}{2c}\big)-\sigma_0\big(\frac{n}{8c}\big)\Big).
\end{aligned}
$$

Now assume that $c$ is even. Then
\begin{align*}
Ad_x: (e,f) \mapsto (e,f+2), && Ad_y: (e,f) \mapsto (e,-f), && Ad_z:
(e,f) \mapsto (-e,-f).
\end{align*}

Thereby for fixed $(a,b,c)$ conjugacy classes are enumerated by
pairs of the following form. The first element of the pair is a
class of odd residues $\{e,-e\}$ modulo $b$, the second element is a
parity of $f$. There are $\frac{b}{4}$ and $\frac{b+2}{4}$ choices
of the first element in the cases $b \equiv 0 \mod 4$ and $b \equiv
2 \mod 4$ respectively. There are $1$ and $2$ choices of parity in
case $a$ is odd and even respectively.

Thus the number of conjugacy classes for a fixed even $c$ is
$$
\begin{aligned}
\sum_{\substack{ab=\frac{n}{c}\\2\mid b}}\left\{
\begin{aligned}
\frac{b+2}{4} \; \text{if} \; 4\nmid b\\
\frac{b}{4}  \quad\;\text{if} \; 4\mid b\\
\end{aligned}\right.\times\left\{\begin{aligned}
1  \;\text{if} \;  2\nmid a\\
2  \;\text{if} \;  2\mid a
\end{aligned}\right.&=\frac{1}{2}\Bigg(\sum_{\substack{ab=\frac{n}{c}\\2\mid
b}}\frac{b}{2}\times\left\{\begin{aligned}
1  \;\text{if} \;  2\nmid a\\
2  \;\text{if} \;  2\mid a
\end{aligned}\right. +\sum_{\substack{ab=\frac{n}{c}\\2\mid b,
4\nmid b}}\left\{
\begin{aligned}
1  \;\text{if} \;  2\nmid a\\
2  \;\text{if} \;  2\mid a
\end{aligned}\right.\Bigg)\\&=
\frac{1}{2}\Big(\sigma_1\big(\frac{n}{2c}\big)+\sigma_1\big(\frac{n}{4c}\big)+\sigma_0\big(\frac{n}{2c}\big)-\sigma_0\big(\frac{n}{8c}\big)\Big).
\end{aligned}
$$
It is noteworthy that we got the same function for both parities of
$c$. One can suggests that there is a single argument to cover both
cases which the authors failed to find.
$$
c_{B_1,B_3}^{case\,3}(n)=\sum_{c \mid
n}\frac{1}{2}\Big(\sigma_1\big(\frac{n}{2c}\big)+\sigma_1\big(\frac{n}{4c}\big)+\sigma_0\big(\frac{n}{2c}\big)-\sigma_0\big(\frac{n}{8c}\big)\Big)=
\frac{1}{2}\Big(\sigma_2\big(\frac{n}{2}\big)+\sigma_2\big(\frac{n}{4}\big)+d_3\big(\frac{n}{2}\big)-d_3\big(\frac{n}{8}\big)\Big).
$$

\subsubsection{Case 5 of \Cref{classification_B_3}\label{case 5}}

This section follows the same scheme as the previous one. Here the
corresponding matrix of a subgroup $\Delta$ has
the form $\begin{pmatrix}  c& e & 0\\0& b & d \\
 0& 0 & a \end{pmatrix}$, where $b$ is odd, $c$ is even and $e$ is an even residue modulo $2b$. So for a
 fixed triple $(a,b,c)$ there are $b$ choices for $e$ and
 $a$ choices for $d$. That is
$$
s_{B_1,B_3}^{case\,5}(n)=\sum_{\substack{abc=n,\\2\nmid b,\, 2 \mid
c}}ab= \sum_{\substack{abc=n,\\ 2 \mid
c}}ab-\sum_{\substack{abc=n,\\ 2 \mid
b,c}}ab=\chi(\frac{n}{2})-2\chi(\frac{n}{4}).
$$

As in previous section, for a fixed triple $(a,b,c)$ the subgroups
$\Delta$ are enumerated by pairs $(d,e)$. Conjugation acts on pairs
in the following way:
\begin{align*}
Ad_x: (d,e) \mapsto (d+2,e), && Ad_y: (d,e) \mapsto (-d,e), && Ad_z:
(d,e) \mapsto (-d,-e).
\end{align*}
Thus conjugacy classes are enumerated by pairs, consisting of the
parity of $d$ and the class $\{e,-e\}$ of residue $e$ modulo $2b$.
There are $1$ and $2$ choices for the first element of such a pair
in case $a$ is odd and even respectively. Also we have
$\frac{b+1}{2}$ choices for the second element.

Then the number of conjugacy classes for a fixed even $c$ is
$$
\begin{aligned}
\sum_{\substack{ab=\frac{n}{c}\\2\nmid b}}\frac{b+1}{2}\times\left\{
\begin{aligned}
1 \; \text{if} \; 2\nmid a\\
2 \; \text{if} \; 2\mid a
\end{aligned}\right.=\frac{1}{2}\Big(\sigma_1\big(\frac{n}{c}\big)-\sigma_1\big(\frac{n}{2c}\big)-2\sigma_1\big(\frac{n}{4c}\big)+\sigma_0\big(\frac{n}{c}\big)-\sigma_0\big(\frac{n}{4c}\big)\Big).
\end{aligned}
$$
Summing over all even values of $c$ one gets
$$
\begin{aligned}
c_{B_1,B_3}^{case\,5}(n)&=\sum_{\substack{c\mid n,\,2\mid
c}}\frac{1}{2}\Big(\sigma_1\big(\frac{n}{c}\big)-\sigma_1\big(\frac{n}{2c}\big)-2\sigma_1\big(\frac{n}{4c}\big)+\sigma_0\big(\frac{n}{c}\big)-\sigma_0\big(\frac{n}{4c}\big)\Big)\\&=\frac{1}{2}\Big(\sigma_2\big(\frac{n}{2}\big)-\sigma_2\big(\frac{n}{4}\big)-2\sigma_2\big(\frac{n}{8}\big)+d_3\big(\frac{n}{2}\big)-d_3\big(\frac{n}{8}\big)\Big).
\end{aligned}
$$
Finally
\begin{align*}
s_{B_1,B_3}(n)&=s_{B_1,B_3}^{case\,3}(n)+s_{B_1,B_3}^{case\,5}(n)=2\chi(\frac{n}{2})-2\chi(\frac{n}{4}),\\
c_{B_1,B_3}(n)&=c_{B_1,B_3}^{case\,3}(n)+c_{B_1,B_3}^{case\,5}(n)=\sigma_2\big(\frac{n}{2}\big)-\sigma_2\big(\frac{n}{8}\big)+d_3\big(\frac{n}{2}\big)-d_3\big(\frac{n}{8}\big).
\end{align*}

\subsection{Case $\Delta \cong \pi_{1}(\mathcal{B}_{2})$\label{B2inB3}}
This section follows the same scheme as \Cref{B1inB3}. Two types of
matrices lead to the same isomorphism type
$\pi_{1}(\mathcal{B}_{2})$, see cases 4 and 6 of
\Cref{classification_B_3}. We consider these cases separately.

\subsubsection{Case 4 of \Cref{classification_B_3}\label{case 4}}
In this case the corresponding matrix of a subgroup $\Delta$ has
the form $\begin{pmatrix}  c& e & f\\0& b & \frac{a}{2} \\
 0& 0 & a \end{pmatrix}$, where $a,b$ are even and $e$ is odd. So for a
 fixed triple $(a,b,c)$ there are $\frac{b}{2}$ choices for $e$ and
 $a$ choices for $f$. That is
$$
s_{B_2,B_3}^{case\,4}(n)=\sum_{\substack{abc=n,\\2\mid
a,b}}\frac{ab}{2} =2\chi(\frac{n}{4}).
$$
Now we enumerate conjugacy classes of subgroups $\Delta$ matching
case 4 of \Cref{classification_B_3}. When replacing a subgroup
$\Delta$ with a conjugated one $\Delta^g, g\in
\pi_{1}(\mathcal{B}_{3})$ the numbers $a,b,c$ are not changed.
Wherein the pair $(e,f)$ is changed in the following way.

First assume that $c$ is odd. Then
\begin{align*}
Ad_x: (e,f) \mapsto (e,f), && Ad_y: (e,f) \mapsto (e+2,-f), &&
Ad_z:(e,f) \mapsto (-e,-f).
\end{align*}
%
%

Recall that $e,f$ first appeared in \Cref{enumeration_B_3} and their
geometrical meaning is the following. Consider the set $\ZZ^2$ and
its subset $Odd=\{(2u+1,v)\}\subset \ZZ^2$. Also consider the group
$G$ of bijections $Odd \mapsto Odd$, generated by $(2u+1,v)\mapsto
(2u+1+b, v+\frac{a}{2})$ and $(2u+1,v)\mapsto (2u+1,v+a)$. The set
of pairs $(e,f)$ is the set of representatives of orbits of $Odd$
under the action of $G$.

So we have to calculate the number of orbits of pairs
$\{(2u+1,v)|u,v\in \ZZ\}$ under the action of the group, generated
by the mappings
\begin{align*}
(2u+1,v)&\mapsto (2u+3,-v), & (2u+1,v)&\mapsto (-2u-1,-v), \\
(2u+1,v)&\mapsto (2u+1+b, v+\frac{a}{2}), & (2u+1,v)&\mapsto
(2u+1,v+a).
\end{align*}
Omitting routine calculations we claim that the number of orbits is
$\frac{a}{2}$ unless simultaneously $b \equiv2 \mod 4$ and $4 \mid
a$; in the latter case the number of orbits is $\frac{a+2}{2}$.

So, the number of conjugacy classes for a fixed odd $c$ is
$$
\begin{aligned}
\sum_{\substack{ab=\frac{n}{c}\\2\mid
a,b}}\frac{a}{2}+\sum_{\substack{ab=\frac{n}{c}\\4\mid
a,\,b\equiv2\textit{mod}4}}1=\sigma_1(\frac{n}{4c})+\sigma_0(\frac{n}{8c})-\sigma_0(\frac{n}{16c}).
\end{aligned}
$$

Now assume that $c$ is even. Then
\begin{align*}
Ad_x: (e,f) \mapsto (e,f+2), && Ad_y: (e,f) \mapsto (e,-f), && Ad_z:
(e,f) \mapsto (-e,-f).
\end{align*}
Following the above argument we have to calculate the number of
orbits of pairs $\{(2u+1,v)|u,v\in \ZZ\}$ under the action of the
group generated by the mappings
\begin{align*}
(u,v)\mapsto (u,v+2), && (u,v)\mapsto (u,-v), && (u,v)\mapsto
(-u,-v), \\ (u,v)\mapsto (u+b, v+\frac{a}{2}), && (u,v)\mapsto
(u,v+a).
\end{align*}
Again omitting calculations, the number of orbits is $\frac{b}{2}$
unless simultaneously $b \equiv2 \mod 4$ and $4 \mid a$; in the
latter case the number of orbits is $\frac{b+2}{2}$.

So, the number of conjugacy classes for a fixed even $c$ is
$$
\begin{aligned}
\sum_{\substack{ab=\frac{n}{c}\\2\mid
a,b}}\frac{b}{2}+\sum_{\substack{ab=\frac{n}{c}\\4\mid
a,\,b\equiv2\textit{mod}4}}1=\sigma_1(\frac{n}{4c})+\sigma_0(\frac{n}{8c})-\sigma_0(\frac{n}{16c}).
\end{aligned}
$$

Summing over possible values of $c$ one gets
$$
c_{B_2,B_3}^{case\,4}(n)=\sum_{c \mid n}
\Big(\sigma_1(\frac{n}{4c})+\sigma_0(\frac{n}{8c})-\sigma_0(\frac{n}{16c})\Big)=\sigma_2(\frac{n}{4})+d_3(\frac{n}{8})-d_3(\frac{n}{16}).
$$

%
%

\subsubsection{Case 6 of \Cref{classification_B_3}\label{case 6}}

In this case the corresponding matrix of a subgroup $\Delta$ has
the form $\begin{pmatrix}  c& e & \frac{a}{2}\\0& b & d \\
 0& 0 & a \end{pmatrix}$, where $b$ is odd, $a,c$ are even and $e$ is an even residue modulo $2b$. So for a
 fixed triple $(a,b,c)$ there are $b$ choices for $e$ and
 $a$ choices for $d$. That is
$$
s_{B_2,B_3}^{case\,6}(n)=\sum_{\substack{abc=n,\\2\nmid b,\, 2 \mid
a,c}}ab= \sum_{\substack{abc=n,\\ 2 \mid
a,c}}ab-\sum_{\substack{abc=n,\\ 2 \mid
a,b,c}}ab=2\chi(\frac{n}{4})-4\chi(\frac{n}{8}).
$$

As in the previous section, for a fixed triple $(a,b,c)$ the
subgroups $\Delta$ are enumerated by pairs $(d,e)$. Conjugation acts
on pairs in the following way:
\begin{align*}
Ad_x: (d,e) \mapsto (d+2,e), && Ad_y: (d,e) \mapsto (-d,e), && Ad_z:
(d,e) \mapsto (-d,-e).
\end{align*}
Thus conjugacy classes are enumerated by invariants, each consisting
of the parity of $d$ and the class $\{e,-e\}$ of residue $e$ modulo
$2b$. Since $a$ is always even, there are $2$ choices for the first
element of such pair. Also we have $\frac{b+1}{2}$ choices for the
second element.

Then the number of conjugacy classes for a fixed even $c$ is
$$
\begin{aligned}
\sum_{\substack{ab=\frac{n}{c}\\2\nmid b,\, 2\mid
a}}(b+1)=\sigma_1\big(\frac{n}{2c}\big)-2\sigma_1\big(\frac{n}{4c}\big)+\sigma_0\big(\frac{n}{2c}\big)-\sigma_0\big(\frac{n}{4c}\big).
\end{aligned}
$$
Summing over all even values of $c$ we get
$$
\begin{aligned}
c_{B_2,B_3}^{case\,5}(n)&=\sum_{\substack{c\mid n,\,2\mid
c}}\Big(\sigma_1\big(\frac{n}{2c}\big)-2\sigma_1\big(\frac{n}{4c}\big)+\sigma_0\big(\frac{n}{2c}\big)-\sigma_0\big(\frac{n}{4c}\big)\Big)\\&=\sigma_2\big(\frac{n}{4}\big)-2\sigma_2\big(\frac{n}{8}\big)+d_3\big(\frac{n}{4}\big)-d_3\big(\frac{n}{8}\big).
\end{aligned}
$$
Finally,
\begin{align*}
s_{B_2,B_3}(n)&=s_{B_2,B_3}^{case\,4}(n)+s_{B_2,B_3}^{case\,6}(n)=4\chi(\frac{n}{4})-4\chi(\frac{n}{8}),\\
c_{B_2,B_3}(n)&=c_{B_2,B_3}^{case\,4}(n)+c_{B_2,B_3}^{case\,6}(n)=2\sigma_2\big(\frac{n}{4}\big)-2\sigma_2\big(\frac{n}{8}\big)+d_3\big(\frac{n}{4}\big)-d_3\big(\frac{n}{16}\big).
\end{align*}

\subsection{Case $\Delta \cong \pi_{1}(\mathcal{B}_{3})$\label{B3inB3}}

In this case the corresponding matrix of a subgroup $\Delta$ has
the form $\begin{pmatrix}  c& e & d\\0& b & d \\
 0& 0 & a \end{pmatrix}$, where $b,c$ are odd, $e$ is even and $0\le e <
 2b$, $0 \le d <a$. Thus
$$
\begin{aligned}
s_{B_3,B_3}(n)&=|\{(a,b,c,d,e)\in \ZZ^6|abc=n,\,2\nmid b,c,\,2\mid
e,\,0\le d <a,\,0\le e <2b\}|\\&=\sum_{\substack{abc=n\\2\nmid
b,c}}ab=\chi(n)-3\chi(\frac{n}{2})+2\chi(\frac{n}{4}).
\end{aligned}
$$

To enumerate the conjugacy classes consider subgroups $\Delta$
corresponding to some fixed factorization $abc=n$. We identify such
subgroups with the pairs $(d,e)$ whence $f=d$. Conjugations act as
follows:
\begin{align*}
Ad_x: (d,e)\mapsto (d+2,e), && Ad_y: (d,e)\mapsto (-d,e+2), && Ad_z:
(d,e)\mapsto (-d,-e).
\end{align*}

Thus for a fixed factorization $abc=n$ the conjugacy classes are
enumerated by the parity of $d$; namely there is only one conjugacy
class if $a$ is odd and two if $a$ is even. Then
$$
c_{B_3,B_3}(n)=\sum_{\substack{abc=n\\2\nmid b,c}}\left\{
\begin{aligned}
1 \; \text{if} \; 2\nmid a\\
2 \; \text{if} \; 2\mid a
\end{aligned}\right. =
d_3(n)-d_3(\frac{n}{2})-d_3(\frac{n}{4})-d_3(\frac{n}{8}).
$$

\subsection{Case $\Delta \cong \pi_{1}(\mathcal{B}_{4})$\label{B4inB3}}
This section is similar to \Cref{B3inB3} with the sole difference
that in all summations the terms corresponding to odd values of $a$
vanish.

Indeed, the corresponding matrix of a subgroup $\Delta$ has
the form $\begin{pmatrix}  c& e & f+\frac{a}{2}\\0& b & d \\
 0& 0 & a \end{pmatrix}$, where $b,c$ are odd, $a,e$ are even and $0\le e <
 2b$, $0 \le d,f <a$ and $d\equiv f \mod a$.  Then
$$
\begin{aligned}
s_{B_4,B_3}(n)&=|\{(a,b,c,d,e,f)\in \ZZ^6|abc=n,\,2\nmid b,c,\,2\mid
a,e,\,0\le d=f <a,\,0\le e <2b\}|\\&=\sum_{\substack{abc=n\\2\nmid
b,c,\, 2\mid
a}}ab=2\chi(\frac{n}{2})-6\chi(\frac{n}{4})+4\chi(\frac{n}{8}),
\end{aligned}
$$
in a similar way
$$
c_{B_4,B_3}(n)=\sum_{\substack{abc=n\\2\nmid b,c,\,2\mid a}}2 =
2d_3(\frac{n}{2})-4d_3(\frac{n}{4})+2d_3(\frac{n}{8}).
$$

\section{The proof of \Cref{th-1-amphidi-} and
\Cref{th-2-amphidi-}}

Most sections of this chapter follow the same logic as their
counterparts in \Cref{chapt_Amphi+}.

\subsection{Case $\Delta \cong \pi_{1}(\mathcal{G}_{1})\cong \ZZ^3$}
As in \ref{ZinB3} we have
\begin{align*}
s_{G_1,B_4}(n)=\omega(\frac{n}{4}), && \text{and} &&
c_{G_1,B_4}(n)=\frac{1}{4}\Big(\omega(\frac{n}{4})+3\sigma_2(\frac{n}{4})+9\sigma_2(\frac{n}{8})\Big).
\end{align*}

\subsection{Case $\Delta \cong \pi_{1}(\mathcal{G}_{2})$}
The proof in this section follows exactly \Cref{G2inB3} until we
explore the condition $h'=h$, which is equivalent to $[z,y]\in J_H$
(for the definition of $H$ and $J_H$  see Notation in \Cref{G2inB3},
$N_y(n)$ is defined in \Cref{Overall_scheme}). Note that in  the
group $\pi_{1}(\mathcal{B}_{4})$ the identity $[z,y]=xy^2$ holds
(compare with $[z,y]=y^2$ in case of $\pi_{1}(\mathcal{B}_{3})$).
Thus $N_y(n)$ appears to be the sum of $2G_{1,1}(k)+G_{total}(k)$
taken over all possible indexes $k=|\langle x,y^2\rangle:H|$
(compare to $2G_{0,1}(k)+G_{total}(k)$ in case of
$\pi_{1}(\mathcal{B}_{3})$). Direct calculations leads to the answer
\begin{align*}
s_{G_2,B_4}(n)&=\omega(\frac{n}{2})-\omega(\frac{n}{4}), \\
c_{G_2,B_4}(n)&=\frac{1}{2}\Big(\sigma_2(\frac{n}{2})+2\sigma_2(\frac{n}{4})-3\sigma_2(\frac{n}{8})+d(\frac{n}{2})-d(\frac{n}{4})-3d(\frac{n}{8})+5d(\frac{n}{16})-2d(\frac{n}{32})\Big).
\end{align*}

\subsection{Case $\Delta \cong \pi_{1}(\mathcal{B}_{1})$}
\subsubsection{Case 3 of \Cref{classification_B_4}}
Discrepancy with \Cref{case 3} is in the action of conjugation.
Namely
\begin{align*}
Ad_x: (e,f) \mapsto (e,f), && Ad_y: (e,f) \mapsto (e+2,-f+1), &&
Ad_z:(e,f) \mapsto (-e,-f+1).
\end{align*}
in case $c$ is odd;
\begin{align*}
Ad_x: (e,f) \mapsto (e,f+2), && Ad_y: (e,f) \mapsto (e,-f), && Ad_z:
(e,f) \mapsto (-e,-f+1).
\end{align*}
in case $c$ is even.

Enumerating the orbits of $(e,f)$ for a fixed value of $c$ we get
$$
\begin{aligned}
\sum_{\substack{ab=\frac{n}{c}\\2\mid b}}\left\{
\begin{aligned}
\frac{a+1}{2}  \;&\text{if} \; 2 \nmid a,\, 4\nmid b\\
\frac{a}{2}  \;&\text{if} \; 2 \mid a,\, 4\nmid b\\
a\quad \;&\text{if} \quad 4\mid b
\end{aligned}\right.&=\frac{1}{2}\Bigg(\sum_{\substack{ab=\frac{n}{c}\\2\mid b}}\left\{
\begin{aligned}
a  \;\text{if} \;  4\nmid b\\
2a  \;\text{if} \;  4\mid b\\
\end{aligned}\right.+\sum_{\substack{ab=\frac{n}{c}\\2\nmid a, b\equiv2\textit{mod}4}}1\Bigg)\\&=
\frac{1}{2}\Big(\sigma_1\big(\frac{n}{2c}\big)+\sigma_1\big(\frac{n}{4c}\big)+\sigma_0\big(\frac{n}{2c}\big)-2\sigma_0\big(\frac{n}{4c}\big)+\sigma_0\big(\frac{n}{8c}\big)\Big)
\end{aligned}
$$
in case of an odd $c$, and
$$
\begin{aligned}
\sum_{\substack{ab=\frac{n}{c}\\2\mid b}}\left\{
\begin{aligned}
\frac{b+2}{4} \; \text{if} \; 2 \nmid a,\,4\nmid b\\
\frac{b}{4}  \quad\;\text{if} \; 2 \nmid a,\,4\mid b\\
\frac{b}{2}  \quad\;\text{if} \; 2 \mid a
\end{aligned}\right.&=\frac{1}{2}\Bigg(\sum_{\substack{ab=\frac{n}{c}\\2\mid
b}}\left\{\begin{aligned}
\frac{b}{2}  \;\text{if} \;  2\nmid a\\
b  \;\text{if} \;  2\mid a
\end{aligned}\right. +\sum_{\substack{ab=\frac{n}{c}\\2\nmid a, b\equiv2\textit{mod}4}}1\Bigg)\\&=
\frac{1}{2}\Big(\sigma_1\big(\frac{n}{2c}\big)+\sigma_1\big(\frac{n}{4c}\big)+\sigma_0\big(\frac{n}{2c}\big)-2\sigma_0\big(\frac{n}{4c}\big)+\sigma_0\big(\frac{n}{8c}\big)\Big).
\end{aligned}
$$
As before we see a coincidence of the functions, derived from
different combinatorics. Summing this over all values of $c$ we get

\begin{align*}
s_{B_3,B_4}^{case\,3}(n)=\sum_{\substack{abc=n,\\2\mid
b}}\frac{ab}{2} =\chi(\frac{n}{2}), && \text{and} &&
c_{B_3,B_4}^{case\,3}(n)=\frac{1}{2}\Big(\sigma_2\big(\frac{n}{2}\big)+\sigma_2\big(\frac{n}{4}\big)+d_3\big(\frac{n}{2}\big)-2d_3\big(\frac{n}{4}\big)+d_3\big(\frac{n}{8}\big)\Big).
\end{align*}

\subsubsection{Case 5 of \Cref{classification_B_4}}
The sole difference from \Cref{case 5} is in the action of the
conjugation
\begin{align*}
Ad_x: (d,e) \mapsto (d+2,e), && Ad_y: (d,e) \mapsto (-d,e), && Ad_z:
(d,e) \mapsto (-d+1,-e).
\end{align*}
This leads to the number of conjugacy classes for a fixed $c$
$$
\begin{aligned}
\sum_{\substack{ab=\frac{n}{c}\\2\nmid
b}}\begin{cases}\frac{b+1}{2}, &2\nmid a \\ b, &2\mid
a\end{cases}=\frac{1}{2}\Big(\sigma_1\big(\frac{n}{c}\big)-\sigma_1\big(\frac{n}{2c}\big)-2\sigma_1\big(\frac{n}{4c}\big)+\sigma_0\big(\frac{n}{c}\big)-2\sigma_0\big(\frac{n}{2c}\big)+\sigma_0\big(\frac{n}{4c}\big)\Big).
\end{aligned}
$$
Summing over all even values of $c$ finishes the job. So
\begin{align*}
s_{B_3,B_4}^{case\,5}(n)&=\sum_{\substack{abc=n,\\2\mid
b}}\frac{ab}{2} =\chi(\frac{n}{2})-2\chi(\frac{n}{4}),
\\
c_{B_3,B_4}^{case\,5}(n)&=\frac{1}{2}\Big(\sigma_2\big(\frac{n}{2}\big)-\sigma_2\big(\frac{n}{4}\big)-2\sigma_2\big(\frac{n}{8}\big)+d_3\big(\frac{n}{2}\big)-2d_3\big(\frac{n}{4}\big)+d_3\big(\frac{n}{8}\big)\Big).
\end{align*}
Finally
\begin{align*}
s_{B_3,B_4}(n)=2\chi(\frac{n}{2})-2\chi(\frac{n}{4}), &&
c_{B_3,B_4}(n)=\sigma_2\big(\frac{n}{2}\big)-\sigma_2\big(\frac{n}{8}\big)+d_3\big(\frac{n}{2}\big)-2d_3\big(\frac{n}{4}\big)+d_3\big(\frac{n}{8}\big).
\end{align*}

\subsection{Case $\Delta \cong \pi_{1}(\mathcal{B}_{2})$}
\subsubsection{Case 4 of \Cref{classification_B_4}}
This section follows \Cref{case 4} exactly up to the explicit
formulas for the action of conjugation which in the present case is
the following.

First assume that $c$ is odd. Then
\begin{align*}
Ad_x: (e,f) \mapsto (e,f), && Ad_y: (e,f) \mapsto (e+2,-f+1), &&
Ad_z:(e,f) \mapsto (-e,-f+1).
\end{align*}

So we have to calculate the number of orbits of pairs
$\{(2u+1,v)|u,v\in \ZZ\}$ under the action of the group generated by
the mappings
\begin{align*}
(2u+1,v)&\mapsto (2u+3,-v+1), & (2u+1,v)&\mapsto (-2u-1,-v+1), \\
(2u+1,v)&\mapsto (2u+1+b, v+\frac{a}{2}), & (2u+1,v)&\mapsto
(2u+1,v+a).
\end{align*}
Then the number of orbits is $\frac{a}{2}$ unless simultaneously
$a,b \equiv2 \mod 4$; in the latter case the number of orbits is
$\frac{a+2}{2}$. Compare with the case of the group
$\pi_{1}(\mathcal{B}_{3})$, where the number of orbits is
$\frac{a}{2}$ unless simultaneously  $b \equiv2 \mod 4$ and $4 \mid
a$; in the latter case the number of orbits is $\frac{a+2}{2}$.

So, the number of conjugacy classes for a fixed odd $c$ is
$$
\begin{aligned}
\sum_{\substack{ab=\frac{n}{c}\\2\mid
a,b}}\frac{a}{2}+\sum_{\substack{ab=\frac{n}{c}\\
a,b\equiv2\textit{mod}4}}1=\sigma_1(\frac{n}{4c})+\sigma_0(\frac{n}{4c})-2\sigma_0(\frac{n}{8c})+\sigma_0(\frac{n}{16c}).
\end{aligned}
$$

Now assume that $c$ is even. Then
\begin{align*}
Ad_x: (e,f) \mapsto (e,f+2), && Ad_y: (e,f) \mapsto (e,-f), && Ad_z:
(e,f) \mapsto (-e,-f+1).
\end{align*}
As in \Cref{case 4}, we have to calculate the number of orbits of
pairs $\{(2u+1,v)|u,v\in \ZZ\}$ under the action of the group
generated by the mappings
\begin{align*}
(2u+1,v)&\mapsto (2u+1,v+2), & (2u+1,v)&\mapsto (2u+1,-v), &
(2u+1,v)&\mapsto(-2u-1,-v+1), \\ (2u+1,v)&\mapsto (2u+1+b,
v+\frac{a}{2}), & (2u+1,v)&\mapsto (2u+1,v+a).
\end{align*}
The number of orbits is $\frac{b}{2}$ unless $a,b \equiv 2\mod 4$,
in the latter case the number is  $\frac{b}{2}+1$. So, the number of
conjugacy classes for a fixed even $c$ is
$$
\begin{aligned}
\sum_{\substack{ab=\frac{n}{c}\\2\mid
a,b}}\frac{b}{2}+\sum_{\substack{ab=\frac{n}{c}\\
a,b\equiv2\textit{mod}4}}1=\sigma_1(\frac{n}{4c})+\sigma_0(\frac{n}{4c})-2\sigma_0(\frac{n}{8c})+\sigma_0(\frac{n}{16c}).
\end{aligned}
$$
Finalizing  our calculations we get
$$
s_{B_2,B_4}^{case\,4}(n)\!=2\chi(\frac{n}{4}),
$$
$$
c_{B_2,B_4}^{case\,4}(n)\!=\!\! \sum_{\substack{c\mid
n}}\!\sigma_1(\frac{n}{4c})\!+\sigma_0(\frac{n}{4c})-2\sigma_0(\frac{n}{8c})+\sigma_0(\frac{n}{16c})\!=\!\sigma_2(\frac{n}{4})+d_3(\frac{n}{4})-2d_3(\frac{n}{8})+d_3(\frac{n}{16}).
$$

\subsubsection{Case 6 of \Cref{classification_B_4}}
The distinction with \Cref{case 6} is in the action of conjugation
\begin{align*}
Ad_x: (d,e) \mapsto (d+2,e), && Ad_y: (d,e) \mapsto (-d,e), && Ad_z:
(d,e) \mapsto (-d,-e).
\end{align*}
This leads to $b$ orbits (compare with $b+1$ in \Cref{case 6}). Thus
\begin{align*}
s_{B_2,B_4}^{case\,6}(n)&=2\chi(\frac{n}{4})-4\chi(\frac{n}{8}), &
c_{B_2,B_4}^{case\,6}(n)&=\sum_{\substack{abc=n,\\2\mid a,c,\,2\nmid
b}}b=\sigma_2\big(\frac{n}{4}\big)-2\sigma_2\big(\frac{n}{8}\big).
\end{align*}
Finally we have
\begin{align*}
s_{B_2,B_4}(n)&=s_{B_2,B_4}^{case\,4}(n)+s_{B_2,B_4}^{case\,6}(n)=4\chi(\frac{n}{4})-4\chi(\frac{n}{8}),\\
c_{B_2,B_4}(n)&=c_{B_2,B_4}^{case\,4}(n)+c_{B_2,B_4}^{case\,6}(n)=2\sigma_2(\frac{n}{4})-2\sigma_2\big(\frac{n}{8}\big)+d_3(\frac{n}{4})-2d_3(\frac{n}{8})+d_3(\frac{n}{16}).
\end{align*}

\subsection{Case $\Delta \cong \pi_{1}(\mathcal{B}_{4})$}
This section is similar to \Cref{B3inB3} with two differences.
First, the summation is taken over all triples $(a,b,c): abc=n$,
where $a,b,c$ are odd. In particular, all sums vanishes if $n$ is
even.

Second, the conjugation acts as follows:
\begin{align*}
Ad_x: (d,e)\mapsto (d+2,e), && Ad_y: (d,e)\mapsto (-d,e+2), && Ad_z:
(d,e)\mapsto (-d+1,-e).
\end{align*}
But this leads to the same result: for a fixed triple of off
positives $(a,b,c)$ there is only one orbit. So we have
\begin{align*}
s_{B_4,B_4}(n)&=\begin{cases}\chi(n) &\text{if $2\nmid n$}\\0
&\text{if $2\mid
n$}\end{cases}=\chi(n)-5\chi\big(\frac{n}{2}\big)+8\chi\big(\frac{n}{4}\big)-4\chi\big(\frac{n}{8}\big),\\
c_{B_4,B_4}(n)&=\begin{cases}d_3(n) &\text{if $2\nmid n$}\\0
&\text{if $2\mid
n$}\end{cases}=d_3(n)-3d_3\big(\frac{n}{2}\big)+3d_3\big(\frac{n}{4}\big)-d_3\big(\frac{n}{8}\big).
\end{align*}

\section*{Appendix\label{Appendix}}
Given a sequence $\{f(n)\}_{n=1}^\infty$, the formal power series
$$
\widehat{f}(s)=\sum_{n=1}^\infty\frac{f(n)}{n^s}
$$
is called a Dirichlet generating function for
$\{f(n)\}_{n=1}^\infty$. To reconstruct the sequence $f(n)$ from
$\widehat{f}(s)$ one can use Perron's formula (\cite{Apostol}, Th.
11.17). Given sequences $f(n)$ and $g(n)$ we call their {\em
convolution} $(f\ast g)(n) = \sum_{k \mid n}f(k)g(\frac{n}{k})$. In
terms of Dirichlet generating series the convolution of sequences
corresponds to the multiplication of generating series
$\widehat{f\ast g}(s)=\widehat{f}(s)\widehat{g}(s)$. For the above
facts see, for example, (\cite{Apostol}, Ch. 11--12).

Here we present the Dirichlet generating functions for the sequences
$s_{H,G}(n)$ and $c_{H,G}(n)$. Since theorems 1--4 provide the
explicit formulas, the remainder is done by direct calculations.

Consider the Riemann zeta function
$\displaystyle\zeta(s)=\sum_{n=1}^{\infty}\frac{1}{n^s}$. Following
\cite{Apostol} note that
\begin{align*}
\widehat{\sigma}_0(s)&=\zeta^2(s), & \widehat{\sigma}_1(s)
&=\zeta(s)\zeta(s-1), &
\widehat{\sigma}_2(s)&=\zeta^2(s)\zeta(s-1),\\
\widehat{d}_3(s)&=\zeta^3(s), & \widehat{\chi}(s)&=
\zeta(s)\zeta(s-1)^2, & \widehat{\omega}(s)&=
\zeta(s)\zeta(s-1)\zeta(s-2).
\end{align*}


Dirichlet generating functions for the sequences provided by
\Cref{enum_all for K} are
\begin{center}
\small{Table~1. Dirichlet generating functions for the sequences,
related to Klein bottle. }
\end{center}
\begin{tabular}{|c|c|}\hline
$\widehat{s}_{\ZZ^2,\pi_1(\mathcal{K})}=2^{-s}\zeta(s)\zeta(s-1)$ &
$\widehat{c}_{\ZZ^2,\pi_1(\mathcal{K})}=2^{-s-1}\zeta(s)\big(\zeta(s-1)+(1+2^{-s})\zeta(s)\big)$\\\hline
$\widehat{s}_{\pi_1(\mathcal{K}),\pi_1(\mathcal{K})}=(1-2^{-s})\zeta(s)\zeta(s-1)$
&
$\widehat{c}_{\pi_1(\mathcal{K}),\pi_1(\mathcal{K})}=(1-2^{-s})(1+2^{-s})\zeta^2(s)$\\\hline
\end{tabular}\medskip

In Tables 2 and 3 we provide the Dirichlet generating functions for
the sequences given by Theorems 1--4. 

\def\formsGIinBIII{$4^{-s}\zeta(s)\zeta(s-1)\zeta(s-2)$}
\def\formcGIinBIII{$4^{-s-1}\zeta(s)\zeta(s-1)\Big(\zeta(s-2)+3(1+3\cdot2^{-s})\zeta(s) \Big)$}
\def\formsGIinBIIII{\formsGIinBIII}
\def\formcGIinBIIII{\formcGIinBIII}
\def\formsGIIinBIII{$2^{-s}(1-2^{-s})\zeta(s)\zeta(s-1)\zeta(s-2)$}
\def\formcGIIinBIII{$ 2^{-s-1}(1-2^{-s})\zeta^2(s)\Big((1+3\cdot2^{-s})\zeta(s-1)+(1+2^{-s}+2^{-2s-1})\zeta(s)\Big)$}
\def\formsGIIinBIIII{\formsGIIinBIII}
\def\formcGIIinBIIII{$ 2^{-s-1}(1-2^{-s})\zeta^2(s)\Big((1+3\cdot2^{-s})\zeta(s-1)+(1+2^{-s}-2^{-2s-1})\zeta(s)\Big)$}
\def\formsBIinBIII{$2^{-s+1}(1-2^{-s})\zeta(s)\zeta(s-1)^2$}
\def\formcBIinBIII{$2^{-s}(1-2^{-s})(1+2^{-s})\zeta^2(s)\Big(\zeta(s-1)+\zeta(s) \Big)$}
\def\formcBIinBIIIalternative{$2^{-s}(1-2^{-s})(1+2^{-s})\zeta^2\zeta(s-1)+2^{-s}(1-2^{-s})(1+2^{-s})\zeta^3$}
\def\formsBIinBIIII{$2^{-s+1}(1-2^{-s})\zeta(s)\zeta(s-1)^2$}
\def\formcBIinBIIII{$2^{-s}(1-2^{-s})\zeta^2(s)\Big((1+2^{-s})\zeta(s-1)+(1-2^{-s})\zeta(s)\Big)$}
\def\formcBIinBIIIIalternative{$2^{-s}(1-2^{-s})(1+2^{-s})\zeta^2\zeta(s-1)+2^{-s}(1-2^{-s})^2\zeta^3$}
\def\formsBIIinBIII{$4^{-s+1}(1-2^{-s})\zeta(s)\zeta(s-1)^2$}
\def\formcBIIinBIII{$4^{-s}(1-2^{-s})\zeta^2(s)\Big(2\zeta(s-1)+(1+2^{-s})\zeta(s) \Big)$}
\def\formcBIIinBIIIalternative{$2\cdot4^{-s}(1-2^{-s})\zeta^2(s)\zeta(s-1)+(1+2^{-s})4^{-s}(1-2^{-s})\zeta^3(s)$}
\def\formsBIIinBIIII{\formsBIIinBIII}
\def\formcBIIinBIIII{$4^{-s}(1-2^{-s})\zeta^2(s)\Big(2\zeta(s-1)+(1-2^{-s})\zeta(s) \Big)$}
\def\formcBIIinBIIIIalternative{$2\cdot4^{-s}(1-2^{-s})\zeta^2(s)\zeta(s-1)+4^{-s}(1-2^{-s})^2\zeta^3(s)$}
\def\formsBIIIinBIII{$(1-2^{-s})(1-2^{-s+1})\zeta(s)\zeta(s-1)^2$}
\def\formcBIIIinBIII{$(1-2^{-s})^2(1+2^{-s})\zeta(s)^3$}
\def\formsBIIIinBIIII{does not exist}
\def\formcBIIIinBIIII{does not exist}
\def\formsBIIIIinBIII{$2^{-s+1}(1-2^{-s})(1-2^{-s+1})\zeta(s)\zeta(s-1)^2$}
\def\formcBIIIIinBIII{$2^{-s+1}(1-2^{-s})^2\zeta(s)^3$}
\def\formsBIIIIinBIIII{$(1-2^{-s})(1-2^{-s+1})^2\zeta(s)\zeta(s-1)^2$}
\def\formcBIIIIinBIIII{$(1-2^{-s})^3\zeta^3(s)$}
\begin{center}
\small{Table~2. Dirichlet generating functions for the sequences
$s_{H,G}(n)$. }
\end{center}
$$
\begin{array}{|c|p{7.5cm}|p{7.5cm}|}\hline
 $\backslashbox{H}{\,G}$& $\pi_1(\mathcal{B}_{3})$ & $\pi_1(\mathcal{B}_{4})$ \\ \hline
\pi_1(\mathcal{G}_{1})  & \formsGIinBIII & \formsGIinBIIII \\
\pi_1(\mathcal{G}_{2})  & \formsGIIinBIII & \formsGIIinBIIII \\
\pi_1(\mathcal{B}_{1})  & \formsBIinBIII & \formsBIinBIIII \\
\pi_1(\mathcal{B}_{2})  & \formsBIIinBIII & \formsBIIinBIIII \\
\pi_1(\mathcal{B}_{3})  & \formsBIIIinBIII & \formsBIIIinBIIII \\
\pi_1(\mathcal{B}_{4})  & \formsBIIIIinBIII & \formsBIIIIinBIIII \\
  \hline
\end{array}
$$
\begin{center}
\small{Table~3. Dirichlet generating functions for the sequences
$c_{H,G}(n)$. }
\end{center}
$$
\begin{array}{|c|p{8.0cm}|p{8.0cm}|}\hline
 $\backslashbox{H}{\,G}$& $\pi_1(\mathcal{B}_{3})$ & $\pi_1(\mathcal{B}_{4})$ \\ \hline
\pi_1(\mathcal{G}_{1})  & \formcGIinBIII & \formcGIinBIIII \\
\pi_1(\mathcal{G}_{2})  & \formcGIIinBIII & \formcGIIinBIIII \\
\pi_1(\mathcal{B}_{1})  & \formcBIinBIII & \formcBIinBIIII \\
\pi_1(\mathcal{B}_{2})  & \formcBIIinBIII & \formcBIIinBIIII \\
\pi_1(\mathcal{B}_{3})  & \formcBIIIinBIII & \formcBIIIinBIIII \\
\pi_1(\mathcal{B}_{4})  & \formcBIIIIinBIII & \formcBIIIIinBIIII \\
  \hline
\end{array}
$$


%
%
%
\def\formcGIinBI{$2^{-s-1}\zeta(s)\zeta(s-1)\Big(\zeta(s-2)+(1+3\cdot2^{-s})\zeta(s) \Big)$}
\def\formcGIinBII{$2^{-s-1}\zeta(s)\zeta(s-1)\Big(\zeta(s-2)+(1-2^{-s}+4\cdot2^{-2s})\zeta(s) \Big)$}
\def\formsBIinBI{$(1-2^{-s})\zeta(s)\zeta^2(s-1)$}
\def\formcBIinBI{$(1-2^{-s})(1+2^{-s})\zeta^2(s)\zeta(s-1)$}
\def\formcBIinBIalternative{\formcBIinBI}
\def\formsBIIinBI{$2^{-s+1}(1-2^{-s})\zeta(s)\zeta^2(s-1)$}
\def\formcBIIinBI{$2^{-s+1}(1-2^{-s})\zeta^2(s)\zeta(s-1)$}
\def\formcBIIinBIalternative{\formcBIIinBI}
\def\formsBIinBII{\formsBIIinBI}
\def\formcBIinBII{\formcBIIinBI}
\def\formcBIinBIIalternative{\formcBIinBII}
\def\formsBIIinBII{$(1-2^{-s})(1-4\cdot 2^{-s}+8\cdot2^{-2s})\zeta(s)\zeta^2(s-1)$}
\def\formcBIIinBII{$(1-2^{-s})(1-3\cdot 2^{-s}+4\cdot2^{-2s})\zeta^2(s)\zeta(s-1)$}
\def\formcBIIinBIIalternative{\formcBIIinBII}

\end{document}